\documentclass[reqno,11pt]{article}
\usepackage{mathrsfs}
\usepackage{amssymb}
\usepackage{amsthm}
\usepackage{amsmath}
\usepackage{amsfonts}
\usepackage{mathtools}
\usepackage{enumerate}
\usepackage{bm}

\usepackage{graphicx}

\usepackage{bbm}
\usepackage{mathrsfs}
\usepackage{amssymb}
\usepackage[thicklines]{cancel}

\usepackage[usenames,dvipsnames]{xcolor}
\usepackage{soul}

\usepackage{txfonts}

\usepackage{psfrag}
\usepackage[numbers,sort&compress]{natbib}
\usepackage{setspace,color,wrapfig}
\usepackage[colorlinks,linkcolor=blue,citecolor=cyan]{hyperref}

\numberwithin{equation}{section}
\newtheorem{theorem}{Theorem}[section]
\newtheorem{lemma}[theorem]{Lemma}
\newtheorem{corollary}[theorem]{Corollary}
\newtheorem{proposition}[theorem]{Proposition}
\newtheorem{problem}[theorem]{Problem}
\theoremstyle{definition}
\newtheorem{definition}[theorem]{Definition}
\newtheorem{assumption}[theorem]{Assumption}
\newtheorem{example}[theorem]{Example}
\newtheorem{remark}[theorem]{Remark}
\theoremstyle{remark}

\allowdisplaybreaks 

\begin{document}

	\title{ Structural stability of cylindrical transonic shocks to the three-dimensional Zeldovich-von Neumann-D\"oring combustion equations   }
	\author{
		Zihao Zhang\thanks{School of Mathematics, Jilin University, Changchun, Jilin Province, 130012, China, and Department of Mathematical Sciences, Seoul National University, Seoul 08826, Republic of Korea. Email: zhangzihao@jlu.edu.cn} \and
		Yan Zhou\thanks{School of Mathematical Sciences, Fudan University, Shanghai, 200433, China. Email: yanzhou@fudan.edu.cn}
	}\date{}
	\maketitle
	\newcommand{\de}{{\mathrm{d}}}
	\def\div{{\rm div\,}}
	\def\curl{{\rm curl\,}}
	\def\th{\theta}
	\newcommand{\ro}{{\rm rot}}
	\newcommand{\sr}{{\rm supp}}
	\newcommand{\sa}{{\rm sup}}
	\newcommand{\va}{{\varphi}}
	\newcommand{\me}{\mathcal{E}}
	\newcommand{\ml}{\mathcal{V}}
	\newcommand{\Md}{\mathcal{D}}
	\newcommand{\mg}{\mathcal{G}}
	\newcommand{\mh}{\mathcal{H}}
	\newcommand{\mf}{\mathcal{F}}
	\newcommand{\ms}{\mathcal{S}}
	\newcommand{\mt}{\mathcal{T}}
	\newcommand{\mn}{\mathcal{N}}
	\newcommand{\mb}{\mathcal{P}}
	\newcommand{\mm}{\mathcal{Q}}
	\newcommand{\mj}{\mathcal{J}}
	\newcommand{\mk}{\mathcal{K}}
	\newcommand{\my}{\mathcal{U}}
	\newcommand{\mw}{\mathcal{W}}
	\newcommand{\mq}{\mathcal{B}}
	\newcommand{\ma}{\mathcal{L}}
	\newcommand{\mc}{\mathcal{C}}
	\newcommand{\mi}{\mathcal{I}}
	\newcommand{\n}{\nabla}
	\newcommand{\e}{\tilde}
	\newcommand{\m}{\Omega}
	\newcommand{\h}{\hat}
	\newcommand{\x}{\bar}
	\newcommand{\q}{{\rm R}}
	\newcommand{\p}{{\partial}}
	\newcommand{\z}{{\varepsilon}}
	\def\be{\begin{eqnarray}}
		\def\ee{\end{eqnarray}}
	\def\p{\partial}
	\def\no{\nonumber}
	\def\eps {\epsilon}
	
	\def\De{\Delta}
	\def\om{\omega}
	\def\Om{\Omega}
	\def\f{\frac}
	\def\th{\theta}
	\def\la{\lambda}
	\def\lab{\label}
	\def\b{\bigg}
	\def\var{\varphi}
	\def\na{\nabla}
	\def\ka{\kappa}
	\def\al{\alpha}
	\def\La{\Lambda}
	\def\ga{\gamma}
	\def\Ga{\Gamma}
	\def\ti{\tilde}
	\def\wti{\widetilde}
	\def\wh{\widehat}
	\def\ol{\overline}
	\def\ul{\underline}
	\def\Th{\Theta}
	\def\si{\sigma}
	\def\Si{\Sigma}
	\def\oo{\infty}
	\def\q{\quad}
	\def\z{\zeta}
	\def\M{\mathbf}
	\def\co{\coloneqq}
	\def\eqq{\eqqcolon}
	\def\bt{\begin{theorem}}
		\def\et{\end{theorem}}
	\def\bc{\begin{corollary}}
		\def\ec{\end{corollary}}
	\def\bl{\begin{lemma}}
		\def\el{\end{lemma}}
	\def\bp{\begin{proposition}}
		\def\ep{\end{proposition}}
	\def\br{\begin{remark}}
		\def\er{\end{remark}}
	\def\bd{\begin{definition}}
		\def\ed{\end{definition}}
	\def\bpf{\begin{proof}}
		\def\epf{\end{proof}}
	\def\bex{\begin{example}}
		\def\eex{\end{example}}
	\def\bq{\begin{question}}
		\def\eq{\end{question}}
	\def\bas{\begin{assumption}}
		\def\eas{\end{assumption}}
	\def\ber{\begin{exercise}}
		\def\eer{\end{exercise}}
	\def\mb{\mathbb}
	\def\mbR{\mb{R}}
	\def\mbZ{\mb{Z}}
	\def\mbf{\mathbf}
	\def\u{\mbf{u}}
	\def\v{\mbf{v}}
	\def\U {\mbf{U}}
	
	\def\x{\mbf{x}}
	\def\mc{\mathcal}
	\def\mcS{\mc{S}}
	\def\lan{\langle}
	\def\ran{\rangle}
	\def\lb{\llbracket}
	\def\rb{\rrbracket}
	\def\div {\mathrm{div}\,}
	\def\curl {\mathrm{curl}\,}
	\pagestyle{myheadings} \markboth{\hfill Transonic   shocks  to the three-dimensional ZND   combustion  model     \hfill}{\hfill Transonic  shocks  to the three-dimensional   ZND   combustion  model \hfill}
	\begin{abstract}
		We address the structural stability of transonic shock solutions to the Zeldovich-von Neumann-D\"oring (ZND) combustion equations in a cylindrical domain under three-dimensional perturbations of the incoming supersonic flow and exit pressure. The perturbed transonic shock is completely free and its strength and position are uniquely determined by the incoming flow and the exit pressure. The solution shares the same regularity for the velocity, the pressure, the entropy and the fraction of unburned gas. The main difficulties are that the shock front is a free boundary  to be determined simultaneously with the subsonic flow and the equations for all states are coupled with each other due to the combustion process. One of the key ingredients of the analysis is to reformulate the three-dimensional ZND combustion equations into a deformation-curl system and four transport equations in terms of the deformation tensor and vorticity. Another one is an elaborate decomposition of the Rankine-Hugoniot conditions so that the shock front is uniquely determined by an algebraic equation and an unusual second order differential boundary condition is obtained on the shock front for the first order nonlocal deformation-curl system, from which an oblique boundary condition can be derived after homogenizing the curl system and introducing the potential function.
	\end{abstract}

	\begin{center}
		\begin{minipage}{5.5in}
		\end{minipage}
	\end{center}
	
	\section{Introduction and  main results }\noindent
	\par In this paper, we are concerned with  the existence and stability of  transonic shock solutions to
	the    three-dimensional steady Zeldovich-von Neumann-D\"oring (ZND)  combustion  equations:
	\begin{equation}\label{1-1}
		\begin{cases}
			\p_{x_1}(\rho u_1)+\p_{x_2}(\rho u_2)+\p_{x_3}(\rho u_3)=0,\\
			\p_{x_1}(\rho u_1^2) + \p_{x_2}(\rho u_1 u_2) + \p_{x_3}(\rho u_1u_3)+\p_{x_1}P=0,\\
			\p_{x_1}(\rho u_1 u_2) + \p_{x_2}(\rho u_2^2) + \p_{x_3}(\rho u_2u_3)+\p_{x_2}P=0,\\
			\p_{x_1}(\rho u_1 u_3) + \p_{x_2}(\rho u_2 u_3) + \p_{x_3}(\rho u_3^2)+\p_{x_3}P=0,\\
			\p_{x_1}(\rho u_1 E+u_1 P)+\p_{x_2}(\rho u_2 E+u_2 P)+\p_{x_3}(\rho u_3 E+u_3 P)=0,\\
			\p_{x_1}(\rho u_1 Y)+\p_{x_2}(\rho u_2 Y)+\p_{x_3}(\rho u_3 Y)=-\rho Y\phi(T),\\
		\end{cases}
	\end{equation}
	in a cylindrical domain
	\begin{equation*}
		\Omega=\left\{{\textbf x}=(x_1,x_2,x_3):r_0<\sqrt{x_1^2+x_2^2}<r_1, \ -\theta_0<\arctan\frac{x_2}{x_1}<\theta_0, \ -L<x_3<L \right\},
	\end{equation*}
	where $0<\theta_0<\frac{\pi}{2}$, $0<r_0<r_1$ and $L>0 $ are  positive constants. In the above system, $\mathbf{u}=(u_1,u_2,u_3)$ denotes the velocity, $\rho$ the density, $P$ the pressure, $Y$ the mass fraction of unburned gas, $E$ the total energy, and $\phi(T)$ the ignition function.
	Typically, the ignition function takes the Arrhenius form
	\begin{equation*}
		\phi(T) =
		\begin{cases}
			T^a e^{-\frac{\mathcal{E}}{\mathcal{R}_0 T}}, & \mathrm{if}\ T > 0, \\
			0, & \mathrm{if}\ T \leq 0,
		\end{cases}
	\end{equation*}
	where $T$ is the temperature, $a$ is a positive constant, $\mathcal{R}_0$ is the universal gas constant, and $\mathcal{E}$ is the activation energy of the mixture. The total energy $E$ is given by
	\begin{equation*}
		E = \frac{1}{2}|\mathbf{u}|^2 + \mathfrak{e} + q_e Y,
	\end{equation*}
	where $\mathfrak{e}$ is the internal energy, and $q_e$ is a constant representing the binding energy of the unburned gas. For convenience, we set $q_e = 1$ throughout this paper. The other two thermodynamic variables are the temperature $T$ and entropy $S$, which are defined via the following thermodynamic relation
	\begin{equation*}
		T \, dS = d\mathfrak{e} - \frac{P}{\rho^2} d\rho.
	\end{equation*}
	Then the pressure $P$ and the internal energy $\mathfrak{e}$ can be regarded as functions of $(\rho, S)$:
	\begin{equation*}
		P = P(\rho, S), \quad \mathfrak{e} = \mathfrak{e}(\rho, S).
	\end{equation*}
	In this paper, we consider an ideal polytropic gas, for which the pressure and internal energy take the forms
	\begin{equation*}
		P = e^S \rho^\gamma \quad \text{and} \quad \mathfrak{e} = \frac{e^S \rho^{\gamma-1}}{\gamma - 1} = T,
	\end{equation*}
	where $\gamma > 1$ is the adiabatic exponent. Denote the Bernoulli function and local sonic speed by
	\begin{equation*}
		B = \frac{1}{2}|\mathbf{u}|^2 + \frac{\gamma P}{(\gamma - 1)\rho} \quad \text{and} \quad c(\rho, S) = \sqrt{\frac{\gamma P}{\rho}} = \sqrt{\gamma(\gamma - 1) T}.
	\end{equation*}
	The system \eqref{1-1} is hyperbolic for supersonic flows ($|\mathbf{u}| > c(\rho, S)$) and of hyperbolic-elliptic mixed type for subsonic flows ($|\mathbf{u}| < c(\rho, S)$).

	\par The ZND combustion model is a classical and well-studied mathematical model for describing the propagation of combustion waves generated by one-step exothermic chemical reactions (see \cite{WF85}). It consists of the Euler equations of gas dynamics coupled with a reaction equation for the combustible gas. A steady detonation wave in the ZND model can be viewed as a special shock solution of the Euler equations with a chemical reaction source term. Consequently, the analytical theories and techniques developed for transonic shock problems in classical gas dynamics provide the theoretical foundation and analytical framework for the study of the well-posedness and stability of combustion waves in the ZND system.
	\subsection{ Brief review}\noindent
	\par Over the past two decades, the stability of two types of transonic shock solutions to the steady Euler equations has been extensively studied. The first type is   transonic  shock solutions in flat nozzles with constant upstream supersonic and downstream subsonic states, where the shock position can be prescribed arbitrarily. The existence and stability of transonic shock solutions for steady potential flow model under various geometric configurations and boundary conditions were established in \cite{CGF03,CFM04,XY05,XY08}. For the full Euler equations,   the existence and stability of the transonic shock for  2-D and 3-D Euler flows with an appropriate end pressure were studied in \cite{CS05,CS08, CY08,XYY09}. In both cases, the existence results require the shock front to pass through  a given point. Without
	such an artificial assumption,  the authors in \cite{FX21}  established the existence and stability of transonic shock solutions
	in an almost flat finite nozzle with the
	exit pressure, where the shock position was uniquely determined by solving an elaborate linear free
	boundary problem. The corresponding results for three-dimensional axisymmetric nozzles  were also obtained in \cite{FG21}.
	\par The second type is   symmetric transonic shock solutons in divergent nozzles, such as radially symmetric solutions in divergent nozzles and spherically symmetric solutions in conical domains, for which the position of the shock front is uniquely determined by the exit pressure.  The authors in \cite{LXY09,LXY13} had proved the existence and stability of the transonic shock solutions in divergent sectors under general perturbation of the nozzle wall and the exit pressure. Similar results for the
	axisymmetric perturbations of the nozzle wall  and the pressure had been
	established in \cite{LXY10,WXX21}. The  stability of spherically
	symmetric subsonic flows and transonic shocks in a spherical shell was established  in \cite{LXY16} by requiring that the background solutions satisfy some ``Structural
	Conditions". Recently, the authors in \cite{WX23,WX25} removed the ``Structural Conditions"
	and established the existence and stability of cylindrical and spherical transonic shocks under three-dimensional perturbations of the incoming flow
	and the exit pressure. In
	\cite{WY24, WZZ25,Z26}, the authors  proved the structural stability of  transonic shock solutions in
	flat nozzles under an external force. It was
	shown that the external force has a stabilization effect on the transonic shock in the nozzle and the
	shock position is uniquely determined. Furthermore, the authors in \cite{WXY21,WYX21} studied  radial symmetric spiral  flows with or without shock in an annulus. Due to nonzero angular velocity, it was found in \cite{WXY21} that besides the  supersonic-subsonic shock in a divergent nozzle, there exists a supersonic-supersonic shock solution, where the downstream state may change smoothly from supersonic to subsonic.
	\par On the other hand, extensive research has also been devoted to initial and boundary value problems for the ZND combustion model. In \cite{CW03}, the global existence of weak solutions to the one-dimensional Cauchy problem was established for initial data with sufficiently small bounded variation. The Cauchy problem for two-dimensional supersonic reacting flows past a Lipschitz wedge with either a large or a small wedge angle was investigated in \cite{CXZ04,CKZ17}. For both the ZND and Chapman--Jouguet combustion models, the finite-time existence of multidimensional unsteady detonation waves was established in \cite{CO07}. The well-posedness of the one-dimensional piston problem for the combustion Euler equations was studied in \cite{HK23,KZ20}. In \cite{FSZ23}, the existence and stability of transonic shock solutions to the steady  ZND combustion system in a two-dimensional finite nozzle were established under the assumption of a small exothermic reaction rate. For the two-dimensional ZND combustion equations in a domain with a Lipschitz wall, the global existence of supersonic entropy solutions containing a strong contact discontinuity was proved in \cite{XZZ18}.  Furthermore, the authors in \cite{GHFWX24} established the existence and uniqueness of supersonic combustion flows with a contact discontinuity in two-dimensional finite nozzles. Most recently, the well-posedness of subsonic  and supersonic combustion flows was investigated in \cite{LS23,LS26,Z25}.
	\par  In this paper, we investigate how the strong coupling among the different states induced by the combustion reaction affects the structural stability of cylindrically symmetric transonic shocks and establish the existence and uniqueness of transonic shock solutions under three-dimensional perturbations of the incoming supersonic flow and the exit pressure.
	\subsection{Mathematical formulation  and main results}\noindent
	\par Introduce the cylindrical coordinates
	$(r, \theta, z)$
	\begin{eqnarray}\label{coor}
		r=\sqrt{x_1^2+x_2^2},\quad \theta=\arctan \frac{x_2}{x_1},\quad z=x_3.
	\end{eqnarray}
	Assume that the velocity,  the pressure, the entropy and the fraction of
	unburned gas are of the form
	\begin{equation*}
		\begin{aligned}
			&{\bf u}({\textbf x})= U_{1}(r,\th,z)\mathbf{e}_r+U_{2}(r,\th,z)\mathbf{e}_\theta+U_{3}(r,\th,z)\mathbf{e}_z, \\
			&S({\textbf x})=S(r,\th,z),\ \ P({\textbf x})=P(r,\th,z),\ \ Y({\textbf x})=Y(r,\th,z),\ \
		\end{aligned}
	\end{equation*}
	with
	\begin{equation*}
		\mathbf{e}_r =
		\begin{pmatrix}
			\cos \th  \\
			\sin \th \\
			0
		\end{pmatrix}, \quad
		\mathbf{e}_{\th } =
		\begin{pmatrix}
			- \sin \th  \\
			\cos \th\\
			0
		\end{pmatrix}, \quad
		\mathbf{e}_z =
		\begin{pmatrix}
			0 \\
			0 \\
			1
		\end{pmatrix}.
	\end{equation*}
	Then
	the steady compressible ZND combustion  equations  \eqref{1-1} in the cylindrical coordinates take the following
	form
	\begin{eqnarray}\label{1-2}
		\begin{cases}
			\begin{aligned}
				&\partial_r(r\rho U_1)+\partial_{\theta}(\rho U_2)+\p_z(r\rho U_3)=0,\\
				&\rho\left(U_1\partial_r +\frac{U_2}{r}\partial_{\theta}+U_3\p_z\right)U_1+\partial_r P-\frac{\rho U_2^2}r=0,\\
				&\rho\left(U_1\partial_r +\frac{U_2}{r}\partial_{\theta}+U_3\p_z\right)U_2
				+\frac{\partial_{\theta}P}{r}+\frac{\rho U_1U_2}{r}=0,\\
				&\rho\left(U_1\partial_r +\frac{U_2}{r}\partial_{\theta}+U_3\p_z\right)U_3+\partial_{z}P=
				0,\\
				&\rho\left(U_1\partial_r +\frac{U_2}{r}\partial_{\theta}+U_3\p_z\right)S=\frac{\rho\phi(T)}{T} Y,\\
				&\rho\left(U_1\partial_r +\frac{U_2}{r}\partial_{\theta}+U_3\p_z\right)Y=-\rho Y\phi(T).
			\end{aligned}
		\end{cases}
	\end{eqnarray}
	The domain $\m$ in the cylindrical coordinates  becomes
	\begin{equation*}
		\Omega=\{(r,\th,z) : r_0<r<r_1,\ (\th,z)\in D\}, \ \ D=(-\theta_0,\theta_0)\times(-L,L).
	\end{equation*}
	The entrance, exit and   boundaries of    $\Omega$   are denoted by
	\begin{equation*}
		\begin{aligned}
			\Gamma_{\text{en}} &= \left\{ (r,\theta,z) : r = r_0,\ -\theta_0 < \theta < \theta_0,\ -L < z < L \right\}, \\
			\Gamma_{\text{ex}} &= \left\{ (r,\theta,z) : r = r_1,\ -\theta_0 < \theta < \theta_0,\ -L < z < L \right\}, \\
			\Gamma_{\theta_0}^\pm &= \left\{ (r,\theta,z) : r_0 \leq r \leq r_1,\ \theta = \pm \theta_0,\ -L \leq z \leq L \right\}, \\
			\Gamma_{L}^\pm &= \left\{ (r,\theta,z) : r_0 \leq r \leq r_1,\ -\theta_0 \leq \theta \leq \theta_0,\ z = \pm L \right\}.
		\end{aligned}
	\end{equation*}
	\par Next, we  consider  a  class of smooth cylindrically symmetric transonic shock solutions.
	More precisely,    the
	flow is described by smooth functions of the form   $(\bar U^{\pm},0,0, \bar \rho^{\pm}, \bar S^{\pm},0)(r)$ on $  r_b<r<r_1$ and $ r_0<r<r_b $ respectively, which solves
	\begin{align}\label{1-15}
		\begin{cases}
			(r\bar\rho^{\pm} \bar U^{\pm})'=0,\\
			\bar\rho^{\pm}\bar U^{\pm}(\bar U^{\pm})'+{(\bar P^{\pm})'} =0,\\
			(\bar S^{\pm})'=0,
		\end{cases}
	\end{align}
	with a discontinuity $r=r_b$ satisfying the physical entropy condition
	$ [\bar P]>0 $
	and the following Rankine-Hugoniot condition
	\begin{align}\label{1-16}
		[\bar \rho \bar U]=[\bar\rho \bar U^2+\bar P]=
		[\bar B]=0,
	\end{align}
	and  the boundary conditions
	\begin{eqnarray}\label{1-17}
		\begin{cases}
			\bar \rho(r_0)=\rho_0^-,\ \bar U(r_0)= U_{0}^-, \ \  \bar S(r_0)= S_0^-,\\
			\bar P(r_1)= P_{e}.
		\end{cases}
	\end{eqnarray}
	Here  $[f]:=f(r_b+)-f(r_b-)$ denotes the jump of $f$ at $r=r_b$ and
	\begin{equation*}
		\bar P^\pm(r)=(e^{\bar S^\pm} (\bar\rho^\pm)^{\gamma})(r),  \ \  \bar B^\pm(r)=\dfrac{1}{2}(\bar U^\pm)^2(r)+\frac{\gamma e^{\bar S^\pm }(\bar\rho^\pm)^{\gamma-1} }{\gamma-1} (r).
	\end{equation*}
	Then it follows from \eqref{1-15}-\eqref{1-17} that
	\begin{equation}\label{1-18}
		\begin{cases}
			(\bar \rho^- \bar U^-)(r)=\frac{r_0 \rho_0^- U_{0}^-}{r}, \ \ \bar B^-(r)=  B_0^- , \ \ \bar S^-(r)=  S_0^-,\ \ r\in(r_0,r_b),\\
			(\bar \rho^+ \bar U^+)(r)=\frac{r_0  \rho_0^-U_{0}^-}{r},\ \
			\bar B^+(r)=  B_0^-,\ \ \bar S^+(r)=  S_0^+,  \ \ r\in(r_b,r_1),
		\end{cases}
	\end{equation}
	where $S_0^+$ is to be determined by \eqref{1-16}.
	
	\begin{proposition}\label{p1}
		Fix $ \gamma> 1 $.   Given constants    $\rho_0^->0$, $ U_{0}^->0 $ and $S_0^->0 $ satisfying
		\begin{equation*}
			(U_{0}^-)^2>\gamma e^{S_0^-}(\rho_0^-)^{\gamma-1},
		\end{equation*}  there exist two positive constants $P_1 $ and $P_0$ such that if the exit pressure   $P_{e}\in(P_1,P_0)$,  the   system \eqref{1-15} with \eqref{1-17} has  a unique piecewise smooth solution
		\begin{equation}\label{1-20}
			\bar{\bm{\Psi}}(r)=
			\begin{cases}
				\bar{\bm{\Psi}}^-(r):=(\bar{U}^{-}(r),0, 0, \bar{P}^{-}(r),{S}_0^{-},0), &  {\rm{if}} \ r_0<r<r_b, \\
				\bar{\bm{\Psi}}^+(r):=(\bar{U}^{+}(r),0, 0, \bar{P}^{+}(r),{S}_0^{+},0), &  {\rm{if}} \ r_b<r<r_1,\\
			\end{cases}
		\end{equation}
		with a shock located at $r\equiv r_b$. Across the shock, the Rankine-Hugoniot condition \eqref{1-16} holds. Furthermore, the solution satisfies $(\bar U^-)^2> c^2(\bar \rho^-, S_0^-)$ for $r\in(r_0,r_b)$, $(\bar U^+)^2<c^2(\bar \rho^+, S_0^+)$ for $r\in(r_b,r_1)$.
	\end{proposition}
	\par The  cylindrically
	symmetric transonic shock solution  $ \bar{\bm{\Psi}}$    will be called the background solution.  Clearly, one can extend the supersonic and subsonic parts
	of $\bar{\bm{\Psi}} $ in a natural way, respectively. For convenience, we still call the
	extended subsonic and supersonic solutions $ \bar{\bm{\Psi}}^+$
	and $  \bar{\bm{\Psi}}^- $. For detailed properties of this cylindrically symmetric transonic shock solution, one may refer to \cite[Section 147]{CF48} or \cite[Theorem 1.1]{xy08b}. The
	goal of this paper is to establish the structural
	stability of the background solution under three-dimensional
	perturbations of the incoming supersonic flow and  the exit pressure.
	\par  Let  the  incoming  supersonic flow at the entrance $ r=r_0 $ be prescribed as
	\begin{equation}\label{1-7-SUP}
		\bm {\Psi}^-(r_0,\th,z)=\bar{\bm{\Psi}}^-(r_0)
		+\sigma \left(U_{1,0}^-, U_{2,0}^-,U_{3,0}^-, P_{0}^-,S_0^-,Y_0^-\right)(\th,z), \quad (\th,z) \in D,
	\end{equation}
	where $ \sigma $ is sufficiently small and $  \left(U_{1,0}^-, U_{2,0}^-,U_{3,0}^-, P_{0}^-,S_0^-,Y_0^- \right)\in \left(C^{2,\alpha}(\overline{ D})\right)^6 $ satisfies the following compatibility conditions
	\begin{equation}\label{1-8-SUP}
		\begin{cases}
			\left(U_{2,0}^-, \p_{\th}^2 U_{2,0}^- \right) (\pm \th_0, z) =  \p_{\th} \left(U_{1,0}^-, U_{3,0}^-, P_0^-,S_0^-,Y_0^- \right) (\pm \th_0, z) = 0, & z \in [- L, L], \\
			\left(U_{3,0}^-, \p_{z}^2 U_{3,0}^- \right) (\th, \pm L) = \p_{z} \left(U_{1,0}^- , U_{2,0}^-, P_0^-,S_0^-,Y_0^- \right) (\th, \pm L) = 0, & \th \in [-\th_0, \th_0].
		\end{cases}
	\end{equation}
	On the nozzle walls, the flow satisfies  the slip boundary condition
	\begin{equation}\label{1-9-SUP}
		\begin{cases}
			U_2(r,\pm \th_0,z) =0, &   (r, z) \in \Gamma_{\theta_0}^\pm,\\
			U_3(r,\th,\pm L) =0, &  (r, \th) \in \Gamma_{L}^\pm.
		\end{cases}
	\end{equation}
	Then the well-posedness    of the ZND combustion equations \eqref{1-2} with \eqref{1-7-SUP} and \eqref{1-9-SUP} for supersonic flows can be obtained by
	the classical theory developed  in \cite{BS07}. More precisely, there exists a constant $\sigma_0 > 0$ depending only on the boundary data, such that for any $0 < \sigma < \sigma_0$, the system \eqref{1-2}  with  \eqref{1-7-SUP} and \eqref{1-9-SUP} has a unique  solution $\bm {\Psi}^- = (U_1^-, U_2^-, U_3^-, P^-,S^-,Y^-) \in \left(C^{2, \alpha} (\overline{\Omega})\right)^6$ satisfying the estimate
	\begin{equation}\label{1-10-SUP}
		\| (U_1^-, U_2^-, U_3^-, P^-,S^-,Y^-) - (\bar{U}^-, 0, 0, \bar{P}^-, S_0^-,0) \|_{C^{2, \alpha} (\overline{\Omega})} \le C_0 \sigma,
	\end{equation}
	and the compatibility conditions
	\begin{equation}\label{1-11-SUP}
		\begin{cases}
			\left(U_2^-, \p_{\th}^2 U_2^- \right) (r, \pm \th_0, z) = \p_{\th} \left(U_1^-, U_3^-, P^-,S^-,Y^- \right) (r, \pm \th_0, z) = 0,   &  (r, z) \in \Gamma_{\theta_0}^\pm,\\
			\left(U_3^-, \p_{z}^2 U_3^- \right) (r,\th,\pm L) = \p_{z} \left(U_1^-, U_2^-, P^-,S^-,Y^- \right)  (r,\th,\pm L) = 0,   & (r, \th) \in \Gamma_{L}^\pm.
		\end{cases}
	\end{equation}
	\par At the exit of the nozzle, the end pressure is prescribed by
	\begin{equation}\label{1-12-sub}
		P(r_1,\theta,z)= P_e+\sigma P_{ex}(\th,z), \quad (\th,z) \in D,
	\end{equation}
	where 
	$ P_{e} (x_2,x_3)\in C^{2,\alpha}(\overline{ D}) $ satisfies the following compatibility conditions
	\begin{equation}\label{1-13-sub}
		\begin{cases}
			\p_{\theta} P_{ex}(\pm \th_0,z)= 0, & z \in [- L, L],\\
			\p_{z} P_{ex}(\th,\pm L) =0, & \th \in [-\th_0, \th_0].
		\end{cases}
	\end{equation}
	In this paper, we want to look for a piecewise smooth solution $ \bm {\Psi}$, which jumps only at a shock front $ \ms=\{(r,\th,z):r=\xi(\th,z),(\th,z)\in D\}$.
	More precisely, $ {\bm \Psi} $ has the  form
	\begin{equation}\label{1-17-sub}
		\bm {\Psi}=
		\begin{cases}
			\bm {\Psi}^-:=(U_1^-,U_2^-,U_3^-,P^-,S^-,Y^-), & {\rm{in}}\ \ \Omega^-=\{(r,\th,z):r_0<r<\xi(\th,z),(\th,z)\in D\},\\
			\bm {\Psi}^+:=(U_1^+,U_2^+,U_3^+,P^+,S^+,Y^+), & {\rm{in}}\  \ \Omega^+=\{(r,\th,z):\xi(\th,z)<r<r_1,  (\th,z)\in D\},\\
		\end{cases}
	\end{equation}
	and satisfies the following  Rankine-Hugoniot conditions on  the shock surface $ \ms $:
	\begin{equation}\label{1-18-sub}
		\begin{cases}
			[\rho U_1]-\frac1\xi\p_{\th}\xi[\rho U_2]-\p_{z}\xi[\rho U_3]=0,\\
			[\rho U_1^2+P]-\frac1\xi\p_{\th}\xi[\rho U_1U_2]-\p_{z}\xi[\rho U_1U_3]=0,\\
			[\rho U_1U_2]-\frac1\xi\p_{\th}\xi[\rho U_2^2+P]-\p_{z}\xi[\rho U_2U_3]=0,\\
			[\rho U_1U_3]-\frac1\xi\p_{\th}\xi[\rho U_2 U_3]-\p_{z}\xi[\rho U_3^2+P]=0,\\
			[B]=0,\\
			[Y]=0.
		\end{cases}
	\end{equation}
	\par The main result in this paper is stated as follows.
	\begin{theorem}\label{main}
		Assume that the compatibility conditions  \eqref{1-8-SUP} and \eqref{1-13-sub}  hold. There exists a suitable constant $\sigma_0 > 0$ depending only on the background solution $\bar{\bm \Psi}$ and the boundary data $(U_{1,0}^-, U_{2,0}^-,U_{3,0}^-, P_{0}^-,S_0^-,Y_0^-,P_e) $ such that if $0 < \sigma < \sigma_0$, the problem \eqref{1-2} with \eqref{1-7-SUP}, \eqref{1-9-SUP} \eqref{1-12-sub} and \eqref{1-18-sub} has a unique solution ${\bm \Psi}^+ = (U_1^+, U_2^+, U_3^+, P^+,S^+,Y^+)  $ with the shock front $\ms: r = \xi (\th,z)$ satisfying the following properties.
		\begin{enumerate}[\rm(1)]
			\item The function $\xi  \in C^{3, \alpha} (\overline{D})$ satisfies the estimate
			\begin{equation}\label{est-xi}
				\|  \xi - r_b \|_{C^{3, \alpha} (\overline{D})} \le C_0 \sigma,
			\end{equation}
			and the compatibility conditions
			\begin{equation}\label{bc-xi}
				\begin{cases}
					\p_{\th} \xi (\pm \th_0, z)  = \p_{\th}^3 \xi(\pm \th_0, z)   = 0,  &  z\in [-L, L], \\
					\p_{z} \xi (\th, \pm L)  = \p_{z}^3 \xi (\th, \pm L) = 0,  &  \th \in [-\th_0, \th_0].
				\end{cases}
			\end{equation}
			\item The solution ${\bm \Psi}^+ = (U_1^+, U_2^+, U_3^+, P^+,S^+,Y^+) \in \left(C^{2, \alpha} (\overline{\Omega^+})\right)^6$ satisfies the entropy condition
			\begin{equation}\label{1-19-dsub}
				S^{+}(\xi(\th,z),\th,z)>S^{-}(\xi(\th,z),\th,z)\  {\rm{on}}\ r=\xi(\th,z), \ \ (\th,z) \in D,
			\end{equation}
			and the estimate
			\begin{equation}\label{est-solu}
				\| {\bm \Psi}^+ -\bar{\bm \Psi}^+ \|_{C^{2, \alpha } (\overline{\Omega^+})} \le C_0 \sigma,
			\end{equation}
			with the compatibility conditions
			\begin{equation}\label{bc-solu}
				\begin{cases}
					(U_2^+, \p_{\th}^2 U_2^+) (r, \pm \th_0, z) = \p_{\th} (U_1^+, U_3^+, P^+,S^+,Y^+) (r, \pm \th_0, z) = 0, &  (r, z) \in [\xi , r_1] \times [-L, L], \\
					(U_3^+, \p_{z}^2 U_3^+) (r, \th, \pm L) = \p_{z} (U_1^+, U_2^+, P^+,S^+,Y^+)(r, \th, \pm L)= 0, &  (r, \th) \in [\xi , r_1] \times [-\th_0, \th_0].
				\end{cases}
			\end{equation}
		\end{enumerate}
		The above positive estimate  constant $C_0 $  depends only on the background solution and the  incoming  supersonic flow and the exit pressure.
	\end{theorem}
	\begin{remark}\label{re1}
		{\it To the best knowledge of the authors, Theorem \ref{main} is   the first existence and uniqueness
			result on the three-dimensional transonic shock solution  to the steady ZND combustion  equations.}
	\end{remark}
	\subsection{Key ideas for the proof }\noindent
	\par  Now we give the key ideas for the proof of the main result in
	this paper. Mathematically, the transonic shock problem can be formulated as a free boundary value problem to the steady ZND combustion equations which is hyperbolic-elliptic mixed in subsonic regions, whose effective decomposition into elliptic and hyperbolic modes is crucial for the solvability of the nonlinear free boundary problem. In terms of the deformation tensor and vorticity developed in \cite{WX19}, we   reformulate the three-dimensional ZND combustion model into four transport equations for the fraction of unburned gas, the Bernoulli quantity, the entropy and the first component of the vorticity and a deformation-curl system for the velocity field.
	Compared with the steady
	Euler equations, the primary challenge  lies in the strong coupling phenomenon due to the combustion process,  in which all physical quantities  are coupled with each other. For the steady ZND model, the chemical reaction leads to the fact that the Bernoulli function and entropy are no longer conserved along streamlines. Instead, their transport equations involve highly nonlinear inhomogeneous source terms, which are jointly driven by the unburned gas mass fraction and the ignition function. Therefore, to avoid the loss of derivatives, we must derive precise estimates of these non-local source terms along the characteristic trajectories, which makes it more complicated to obtain the structural stability of the shock solution by the iteration method than that in \cite{WX23}.
	\par To solve this free boundary problem, our first step is to  reformulate the Rankine-Hugoniot jump conditions into  \eqref{rh2} and \eqref{rh10}, which
	\begin{itemize}
		\item shows the mechanism of determining uniquely the shock front by an algebraic equation whose principal part is a scalar multiple of the difference of the radial velocity;
		\item provides the boundary data on the shock front for the entropy, the Bernoulli quantity, the fraction of unburned gas and the first component of the vorticity (see \eqref{rh18}), which represent the hyperbolic modes;
		\item provides an unusual second order differential boundary condition on the shock front to the first order deformation-curl system for the velocity field (see \eqref{shock19}), which represents the elliptic modes.
	\end{itemize}
	\par Next, we use the  iteration method to obtain the structural stability of  cylindrical transonic shock solutions. Different from \cite{WX23}, the combustion reaction procedure leads to the strong coupled phenomena for all states.  We will develop a new iteration scheme based on an observation that the source terms for the equations of $(U_1^+,U_2^+,U_3^+,B^+,S^+)$ contain $ Y^+$, and $ Y^+$  can be regarded as a higher-order term. To this end, we first solve the transport equation for $ Y^+$, substitute its expression into the transport equations for the Bernoulli quantity and the entropy, and then solve them. Using their representation formulas, we can further solve the transport equation for the first component of vorticity. Furthermore,  it must be verified that the compatibility conditions are preserved on the intersection of the shock front and the cylinder walls.  Then we derive the deformation-curl system for the velocity field with nonlocal terms involving the trace of the radial velocity on the shock front.
	
	Due to linearization, the vector field
	may not be divergence free and thus the solvability condition of the curl system
	does not hold in general. Therefore, we introduce a new unknown function $\Pi$ that satisfies Dirichlet boundary conditions. It can be shown that $\Pi$ exists and is unique (see Step 5.1), and that $\Pi$ corresponding to the fixed point of the operator vanishes automatically (see Step 7). Subsequently, we solve the div-curl system with homogeneous normal boundary conditions (see Step 5.2). Using  the auxiliary function, we homogenize the curl system and introduce a potential function, thereby transforming the nonlocal boundary condition into a local one. We observe that the second-order elliptic equation for the potential function contains the nonlocal term $\partial_{y_1} \varphi (r_b, y')$. A key point is then to derive an oblique derivative boundary condition that changes this nonlocal term into $\varphi (r_b, y')$, which further simplifies the proof of unique solvability for the system (see Step 5.3).
	\par This paper is organized as follows. In Section \ref{reformulation}, we address the intrinsic hyperbolic-elliptic coupling structure by employing the deformation and vorticity, and reformulate the Rankine-Hugoniot conditions. Then we introduce a coordinate transformation to reduce  the transonic shock problem into the fixed boundary value problem. In Section \ref{iteration}, we design an iteration scheme to prove Theorem \ref{main}.
	
	\section{The reformulation of the transonic shock problem}\label{reformulation}\noindent
	\par In this section, we first  reformulate  the steady ZND combustion equations in terms of  the deformation and vorticity, and then introduce a coordinate transformation to fix the free boundary.
	\subsection{  The reformulation of  the ZND combustion equations }\noindent
	\par It is well-known that the steady  ZND combustion  equations are hyperbolic-elliptic mixed in the subsonic region. It is crucial to find an effective way to decompose the hyperbolic and elliptic modes for further mathematical analysis. Here we deal with the hyperbolic-elliptic coupled structure in terms of the deformation tensor and vorticity  developed in \cite{WX19}.
	\par The fraction of unburned
	gas, the Bernoulli quantity and the entropy    are governed by the following transport equations:
	\begin{align}\label{2-1}
		&\left(\p_r+\frac{U_2}{U_1}\frac{1}{r}\p_{\theta} +\frac{U_3}{U_1}\p_{z}\right)Y=-\frac{{\phi(T)} Y}{U_1},\\\label{2-2}
		&\left(\p_r+\frac{U_2}{U_1}\frac{1}{r}\p_{\theta} +\frac{U_3}{U_1}\p_{z}\right) B =\frac{\phi(T)Y}{U_1} , \\\label{2-3}
		&\left(\p_r+\frac{U_2}{U_1}\frac{1}{r}\p_{\theta} +\frac{U_3}{U_1}\p_{z}\right)S=\frac{\phi(T)Y}{TU_1} .
	\end{align}
	It  follows from the third and fourth equations in \eqref{1-2} that
	\begin{equation}\label{om1}
		\begin{cases}
			U_1 \omega_3 - U_3 \omega_1 + \frac1r\p_{\th} B-\frac{e^S\rho^{\gamma-1}}{\gamma-1} \frac1r\p_{\th}S  = 0, \\
			U_2 \omega_1 - U_1 \omega_2 + \p_{z} B-\frac{e^S\rho^{\gamma-1}}{\gamma-1} \p_{z}S = 0.
		\end{cases}
	\end{equation}
	Thus one  obtains
	\begin{equation}\label{om2}
		\begin{cases}
			\begin{aligned}
				\omega_2 &= \frac {U_2 \omega_1 + \p_{z}B}{U_1}-\frac{e^S\rho^{\gamma-1}}{(\gamma-1)U_1} \p_{z}S , \\
				\omega_3 &= \frac {U_3 \omega_1 - \frac1r\p_{\th} B}{U_1}+\frac{e^S\rho^{\gamma-1}}{(\gamma-1)U_1} \frac1r\p_{\th}S .\end{aligned}
		\end{cases}
	\end{equation}
	Note that
	\begin{equation}\label{om0}
		\div \curl {\bf u}=\p_r\om_1+\frac{1}{r}\p_{\theta}\om_2+\p_{z}\om_3 +\frac{\omega_1}{r} = 0.
	\end{equation}
	Substituting \eqref{om2}  into the equation  \eqref{om0} yields that
	\begin{equation}\label{om3}
		\begin{aligned}
			& \left( \partial_r + \frac{U_2}{r U_1} \partial_\theta + \frac{U_3}{U_1} \partial_z \right) \omega_1
			+ \left( \frac{1}{r} + \frac{1}{r} \partial_\theta \left( \frac{U_2}{U_1} \right) + \partial_z \left( \frac{U_3}{U_1} \right) \right) \omega_1
			+ \frac{1}{r} \partial_\theta \left( \frac{1}{U_1} \right) \partial_z B
			- \partial_z \left( \frac{1}{U_1} \right) \frac{1}{r} \partial_\theta B \\
			& - \partial_\theta \left( \frac{e^S \rho^{\gamma-1}}{(\gamma-1) U_1} \right) \frac{1}{r} \partial_z S
			+ \partial_z \left( \frac{e^S \rho^{\gamma-1}}{(\gamma-1) U_1} \right) \frac{1}{r} \partial_\theta S = 0.
		\end{aligned}
	\end{equation}
	\par Next, using the  Bernoulli quantity, the density can be represented as
	\begin{equation}\label{rho1}
		\rho = \mh (B, S,U_1,U_2,U_3) = \left( \frac {\gamma -1}{\gamma e^S } \right)^{\frac 1 {\gamma - 1}} \left( B - \frac 12 \left(U_1^2+U_2^2+U_3^2 \right) \right)^{\frac 1 {\gamma -1}}.
	\end{equation}
	Substituting \eqref{rho1} into the continuity equation and combining with \eqref{2-1}-\eqref{2-3} lead to
	\begin{equation}\label{rho2}
		\begin{aligned}
			&
			(c^2 (\mh,S ) - U_1^2 ) \p_{r} U_1 + (c^2 (\mh,S ) - U_2^2 ) \frac{\p_\th U_2}{r} + (c^2 (\mh,S ) - U_3^2) \p_{z} U_3 \\
			&\ + \frac{c^2(\mh,S) U_1}{r} =
			U_1 (U_2 \p_{r} U_2 + U_3 \p_{r} U_3  )
			+ \frac{U_2}{r} (U_1 \p_{\th} U_1 + U_3 \p_{\th} U_3 )\\
			&\ + U_3 (U_1 \p_{z} U_1 + U_2 \p_{z} U_2  )+(\gamma-1)\phi(T) Y.
		\end{aligned}
	\end{equation}
	Together with the vorticity equations gives a deformation-curl system
	\begin{equation}\label{dc}
		\begin{cases}
			\begin{aligned}
				& \bigl( c^2(\mh, S) - U_1^2 \bigr) \partial_r U_1 + \bigl( c^2(\mh, S) - U_2^2 \bigr) \frac{\partial_\theta U_2}{r} + \bigl( c^2(\mh, S) - U_3^2 \bigr) \partial_z U_3 \\
				&  + \frac{c^2(\mh, S) U_1}{r} = U_1 \bigl( U_2 \partial_r U_2 + U_3 \partial_r U_3 \bigr) + \frac{U_2}{r} \bigl( U_1 \partial_\theta U_1 + U_3 \partial_\theta U_3 \bigr) \\
				&  + U_3 \bigl( U_1 \partial_z U_1 + U_2 \partial_z U_2 \bigr) + (\gamma-1) \phi(T) Y, \\
				& \frac{1}{r} \partial_\theta U_3 - \partial_z U_2 = \omega_1, \\
				& \partial_z U_1 - \partial_r U_3 = \omega_2, \\
				& \partial_r U_2 - \frac{1}{r} \partial_\theta U_1 + \frac{U_2}{r} = \omega_3.
			\end{aligned}
		\end{cases}
	\end{equation}
	\begin{lemma}\label{equiv1}
		Assume that a  smooth vector function $( U_1,U_2,U_3,P,S,Y)\in \left(C^1(\overline{\m})\right)^6$ defined on the domain $\m$ does not contain the vacuum (i.e. $\rho > 0$ in $\m$) and the radial velocity $U_1>0 $ in $\m$. Then the following two statements are equivalent.
		\begin{enumerate}[(i)]
			\item $(U_1,U_2,U_3,P,S,Y)$ satisfies the steady  {\rm{ZND}} combustion equations \eqref{1-2} in $\m$;
			\item $( U_1,U_2,U_3,B,S,Y)$ satisfies  \eqref{2-1}, \eqref{2-2}, \eqref{2-3} , \eqref{om2} and \eqref{dc} in $\m$.
		\end{enumerate}
	\end{lemma}
	Next, to simplify the notations, we neglect the superscript ``+"  in $(U_1^+,U_2^+,U_3^+,B^+,S^+,Y^+)$. For the perturbation of subsonic flows, set
	\begin{gather*}
		V_1(r, \theta, z) = U_1(r, \theta, z) - \bar{U}^+(r), \quad
		V_2(r, \theta, z) = U_2(r, \theta, z), \quad
		V_3(r, \theta, z) = U_3(r, \theta, z), \\
		V_4(r, \theta, z) = S(r, \theta, z) - S_0^+, \quad
		V_5(r, \theta, z) = B(r, \theta, z) - B_0^-, \quad
		V_6(r, \theta, z) = Y^+(r, \theta, z), \\
		\mathbf{V} = (V_1, \cdots, V_6), \quad
		V_7(\theta, z) = \xi(\theta, z) - r_b.
	\end{gather*}
	Then the density, the pressure and the  temperature  can be expressed as
	\begin{equation}\label{2-11}
		\begin{cases}
			\rho
			=\left(\frac{\gamma-1}{\gamma e^{(S_0^++V_4)}}\right)^{\frac{1}{\gamma-1}}
			\left(B_0^-+V_5-\frac12(\bar{U}^++V_1)^2-\frac12V_2^2
			-\frac12 V_3^2\right)^{\frac{1}{\gamma-1}},\\
			P
			=\left(\frac{\gamma-1}{\gamma}\right)^{\frac{\gamma}{\gamma-1}} \left( \frac{1}{e^{(S_0^++V_4)}}\right)^{\frac{1}{\gamma-1}}
			\left(B_0^-+V_5-\frac12(\bar{U}^++V_1)^2-\frac12 V_2^2
			-\frac12 V_3^2\right)^{\frac{\gamma}{\gamma-1}}, \\
			T=\frac1{\gamma}\left( B_0^-+ V_5 -\frac{1}{2} (  \bar{U}^++V_1)^2-\frac{1}{2} V_2^2-\frac{1}{2} V_3^2\right).
		\end{cases}
	\end{equation}
	Furthermore, one can rewrite the equations \eqref{2-1}-\eqref{2-3}, \eqref{om2}, \eqref{om3} and \eqref{rho2} in terms of ${\bf V}$ as follows. The equations for the hyperbolic quantities   $V_6$, $V_5$ and $V_4$ are
	\begin{align}\label{2-1-a}
		&\left( \p_{r} + \frac{V_2}{\bar U^++ V_1}\frac1r\p_{\th} + \frac{V_3}{\bar U^++ V_1} \p_{z}\right)V_6=  -\frac{{\phi(T)} V_6}{\bar U^++ V_1},\\\label{2-2-b}
		&\left( \p_{r} + \frac{V_2}{\bar U^++ V_1}\frac1r\p_{\th} + \frac{V_3}{\bar U^++ V_1} \p_{z}\right)V_5=  \frac{{\phi(T)} V_6}{\bar U^++ V_1}  , \\\label{2-3-c}
		&\left( \p_{r} + \frac{V_2}{\bar U^++ V_1}\frac1r\p_{\th} + \frac{V_3}{\bar U^++ V_1} \p_{z}\right)V_4=\frac{{\phi(T)} V_6}{T(\bar U^++ V_1)}   .
	\end{align}
	The equations for the vorticity are
	\begin{equation}\label{om3-1}
		\begin{aligned}
			&\left( \p_{r} + \frac{V_2}{\bar U^++ V_1}\frac1r\p_{\th} + \frac{V_3}{\bar U^++ V_1} \p_{z} \right) \omega_1
			+ \left(\frac{1}{r}+\frac{1}{r}\p_{\theta}\left(\frac{V_2}{\bar U^++ V_1}\right)+\p_{z}\left(\frac{V_3}{\bar U^++ V_1}\right)
			\right)\omega_1\\
			&+\frac{1}{r}\p_{\theta}\left(\frac{1}{\bar U^++ V_1}\right)\p_{z}V_5
			-\p_{z}\left(\frac{1}{\bar U^++ V_1} \right)\frac{1}{r}\p_{\theta}V_5
			-  \frac1r\p_{\th} \left(\frac{e^{( S_0^++V_4)}\rho^{\gamma-1}}{\bar U^++V_1} \right) \frac{\p_{z} V_4}{\gamma-1}\\
			&+\frac1r\p_{z} \left( \frac{e^{( S_0^++V_4)}\rho^{\gamma-1}}{\bar U^++V_1} \right)\frac{ \p_{\th} V_4}{\gamma-1}= 0,
		\end{aligned}
	\end{equation}
	and
	\begin{equation}\label{om2-1}
		\begin{cases}
			\begin{aligned}
				\omega_2 &= \frac {V_2 \omega_1 + \p_{z}V_5}{\bar U^++V_1}-\frac{e^{( S_0^++V_4)}\rho^{\gamma-1}}{(\gamma-1)(\bar U^++V_1)} \p_{z}V_4 , \\
				\omega_3 &= \frac {V_3 \omega_1 - \frac1r\p_{\th} V_5}{\bar U^++V_1}+\frac{e^{( S_0^++V_4)}\rho^{\gamma-1}}{(\gamma-1)(\bar U^++V_1)} \frac1r\p_{\th}V_4 .
			\end{aligned}
		\end{cases}
	\end{equation}
	It follows from \eqref{rho2} that
	\begin{equation}\label{rho2-1}
		\begin{aligned}
			d_1(r)\p_{r} V_1 + \frac1r\p_{\th} V_2 +   \p_{z} V_3+\frac{ V_1}r+d_2(r)V_1=d_3(r)V_5+d_4(r)V_6
			+F({\bf V}),
		\end{aligned}
	\end{equation}
	where
	\begin{equation*}
		\begin{aligned}
			&d_1(r)=1-(\bar M^+)^2, \quad
			(\bar M^+)^2(r)=\frac{(\bar U^+)^2}{(\bar c^+)^2}, \quad (\bar c^+)^2(r)=\gamma e^{S_0^+}(\bar\rho^+)^{\gamma-1},\\
			&d_2(r)=
			-\frac{(\gamma+1)\bar{U}^+(\bar{U}^+)'}{(\bar c^+)^2}
			-\frac{\gamma-1}{r}(\bar M^+)^2=\frac{(\bar M^+)^2 \left(2+(\gamma-1)(\bar M^+)^2\right)}{r  \left( 1-(\bar M^+)^2 \right)}, \\
			&d_3(r)=-\frac{(\gamma-1)\left((\bar U^+)'+\frac{(\bar U^+)}r\right)}{(\bar c^+)^2},\ \  d_4(r)=\frac{(\gamma-1)\phi(\bar T^+)}{(\bar c^+)^2}, \bar T(r)=\frac1{\gamma}\left( B_0^- -\frac{1}{2}(  \bar{U}^+)^2\right),\\
			&F({\bf V})=-\frac{(\gamma-1)(\p_r V_1+\frac{V_1}{r})}{(\bar c^+)^2} V_5+ \frac{(\bar{U}^+)'+\p_r V_1}{(\bar c^+)^2}\left(\frac{\gamma+1}{2} V_1^2+\frac{\gamma-1}{2}(V_2^2+V_3^2)\right)\\
			&\quad \qquad+\frac{(\gamma-1)(\bar{U}^++V_1)}{2r (\bar c^+)^2}\sum_{j=1}^3 V_j^2+\frac{(\gamma+1)\bar{U}^+V_1 \p_r V_1+ (\gamma-1)\bar{U}^+ \frac{V_1^2}{r}}{(\bar c^+)^2}\\
			&\quad \qquad- \frac{1}{(\bar c^+)^2}\bigg((\gamma-1)V_5-\frac{\gamma-1}{2}\sum_{j=1}^3 V_j^2-(\gamma-1)\bar{U}^+V_1\bigg)\left(\frac{1}{r}\p_{\theta} V_2+\p_{z} V_3\right)\\
			&\quad \qquad
			+\frac{1}{(\bar c^+)^2}\bigg(V_2^2\frac{1}{r}\p_{\theta} V_2 +V_3^2 \p_{z}V_3\bigg)+ \frac{\bar{U}^++V_1}{\bar{c}^2}(V_2\p_r V_2+V_3\p_r V_3)\\
			&\quad \qquad+\frac{V_2}{(\bar c^+)^2}\left((\bar{U}^++V_1)\frac{1}{r}\p_{\theta} V_1 +\frac{V_3}{r}\p_{\theta} V_3\right)+ \frac{V_3}{(\bar{c}^+)^2}\left((\bar{U}^++V_1)\p_{z} V_1 +V_2\p_{z} V_2\right)\\
			&\quad \qquad+\frac{1}{(\bar c^+)^2}\left((\gamma-1)\phi( T)-(\gamma-1)\phi(\bar T^+)\right)V_6.
		\end{aligned}
	\end{equation*}
	\subsection{The linearization  of the Rankine-Hugoniot conditions }\noindent
	\par Due to the mixed elliptic-hyperbolic structure of the steady  ZND combustion  equations in the subsonic region, it is important to find an elaborate decomposition of the Rankine-Hugoniot conditions and yield suitable boundary conditions for the hyperbolic and elliptic modes. To this end,
	we first linearize the shock front. It follows from the third and fourth equations in \eqref{1-18-sub} that
	\begin{equation}\label{rh1}
		\frac{\p_{\th} \xi}{\xi} = \frac {\ma_2(\xi, \th,z)}{\ma (\xi, \th,z)}, \quad
		\p_{z} \xi =\frac {\ma_3(\xi, \th,z)}{\ma (\xi, \th,z)},
	\end{equation}
	with
	\begin{equation*}
		\begin{aligned}
			& \ma (\xi, \th,z) = [\rho U_2^2 + P] [\rho U_3^2 + P ] - ([\rho U_2 U_3])^2, \\
			& \ma_2 (\xi, \th,z) = [\rho U_3^2 + P ] [\rho U_1 U_2 ] - [\rho U_1 U_3] [\rho U_2 U_3], \\
			& \ma_3 (\xi, \th,z) = [\rho U_2^2 + P] [\rho U_1 U_3] - [\rho U_1 U_2] [\rho U_2 U_3].
		\end{aligned}
	\end{equation*}
	Then  \eqref{rh1} can be rewritten as
	\begin{equation}\label{rh2}
		\begin{cases}
			\p_{\th} \xi (\th,z) = r_bb_0 V_2 (\xi ,\th,z) + r_b g_2 \left({\bm \Psi}^- (\xi ,\th,z) - \bar{\bm \Psi}^-(\xi), {\bf V} (\xi ,\th,z), V_7 (\th,z) \right), \\
			\p_{z} \xi (\th,z) = b_0 V_3 (\xi ,\th,z) + g_3 \left({\bm \Psi}^- (\xi ,\th,z) - \bar{\bm \Psi}^-(\xi), {\bf V} (\xi ,\th,z), V_7 (\th,z)\right), \\
		\end{cases}
	\end{equation}
	where
	\begin{equation*}
		\begin{aligned}
			&b_0 = \frac {(\bar{\rho}^+ \bar{U}^+)(r_b)  }{[\bar{P}(r_b) ]} > 0,\\
			&g_2 \left({\bm \Psi}^- (\xi ,\th,z) - \bar{\bm \Psi}^-(\xi), {\bf V} (\xi ,\th,z), V_7 (\th,z)\right) =\frac1{ r_b}\left( \frac {\xi\ma_2(\xi ,\th,z)}{\ma(\xi ,\th,z)} - r_bb_0 V_2 (\xi ,\th,z)\right), \\
			&g_3 \left({\bm \Psi}^- (\xi ,\th,z) - \bar{\bm \Psi}^-(\xi), {\bf V} (\xi ,\th,z), V_7 (\th,z)\right)= \frac {\ma_3(\xi ,\th,z)}{\ma(\xi ,\th,z)} - b_0 V_3 (\xi ,\th,z).
		\end{aligned}
	\end{equation*}
	The functions $g_i\ (i = 2,3)$ can be regarded as error terms which are bounded by
	\begin{equation}\label{rh3}
		|g_i| \le C_0 \left(|{\bm \Psi}^- (\xi ,\th,z) - \bar{\bm \Psi}^-(\xi)| + |{\bf V} (\xi ,\th,z)|^2 + |V_7 (\th,z)|^2 \right), \quad i=2,3.
	\end{equation}
	\par Next, we derive the boundary conditions  on the shock front. It follows from \eqref{1-18-sub} and \eqref{rh1} that
	\begin{equation}\label{rh4}
		\begin{cases}
			\begin{aligned}
				&[\rho U_1] = \frac {[\rho U_2] \ma_2 + [\rho U_3] \ma_3}{\ma},  \\
				&[\rho U_1^2 + P ] = \frac {[\rho U_1 U_2] \ma_2 + [\rho U_1 U_3] \ma_3}{\ma},\\
				& [B]=
				[Y]=0.
			\end{aligned}
		\end{cases}
	\end{equation}
	Denote $ V_0 (r,\th,z) =\rho(r,\th,z)-\bar \rho^+(r) $. Then one can conclude from \eqref{rh4} that at $(\xi,\th,z)$, it holds that
	\begin{equation}\label{rh5}
		\begin{cases}
			\begin{aligned}
				&\bar{\rho}^+(r_b) V_1+{\bar U}^+(r_b)V_0  = R_{11}(\bm{\Psi}^-(\xi,\th,z) - \bar {\bm{\Psi}}^-(\xi),{\bf V}(\xi,\th,z), V_7 (\th,z)),\\
				&2(\bar{\rho}^+ \bar{U}^+)(r_b)V_1+ \left((\bar U^+)^2(r_b)+(\bar c^+)^2(r_b)\right)V_0+e^{S_0^+}(\bar{\rho}^+)^{\gamma}(r_b) V_4\\
				&=-\frac{1}{r_b}[\bar{P}(r_b)] V_7+ R_{12}(\bm{\Psi}^-(\xi,\th,z) - \bar {\bm{\Psi}}^-(\xi),{\bf V}(\xi,\th,z), V_7 (\th,z)),\\
				&\bar{U}_{1}^+(r_b)V_1+ \frac{(\bar c^+)^2(r_b)}{\bar\rho^+(r_b)} V_0+ \frac{\gamma  e^{S_0^+}(\bar\rho^+)^{\gamma-1}(r_b)}{\gamma-1} V_4\\
				&=R_{13}(\bm{\Psi}^-(\xi,\th,z) - \bar {\bm{\Psi}}^-(\xi),{\bf V}(\xi,\th,z), V_7  (\th,z)),\\
			\end{aligned}
		\end{cases}
	\end{equation}
	where
	\begin{equation*}
		\begin{aligned}
			& R_{11}(\bm{\Psi}^-(\xi,\th,z) - \bar {\bm{\Psi}}^-(\xi),{\bf V}(\xi,\th,z), V_7 (\th,z))\\
			&=-[\bar{\rho}\bar{ U}_{1}](\xi)
			-(V_1 +\bar U_{1}^+(\xi)
			- \bar U_{1}^+(r_b))V_0(\xi,\th,z)-(\bar{\rho}^+(\xi)-\bar{\rho}^+(r_b))V_1(\xi,\th,z)\\
			&\quad
			+ (\rho^- U_1^- - \bar{\rho}^- \bar{ U}^- ) (\xi, \theta, z)
			+ \frac{[\rho U_2] \ma_2+[\rho U_3] \ma_3}{\ma},\\
			&R_{12}(\bm{\Psi}^-(\xi,\th,z) - \bar {\bm{\Psi}}^-(\xi),{\bf V}(\xi,\th,z), V_7 (\th,z))\\
			&=-\left([\bar\rho \bar{U}_1^2+ \bar{P}](\xi)-\frac{1}{r_b}[\bar{P}(r_b)] V_7\right)+ (\rho^- (U_1^-)^2 +P^-)(\xi,\th,z)-(\bar{\rho}^-(\bar{U}_{1}^-)^2+\bar{P}^-)(\xi)\\
			&\quad-\left((\rho U_1^2 +P)(\xi,\th,z)-(\bar{\rho}^+(\bar{U}^+)^2+\bar{P}^+)
			(\xi)-2(\bar{\rho}^+ \bar{U}^+)(r_b)V_1\right) \\
			&\qquad + \left((\bar U^+)^2(r_b) +(\bar c^+)^2(r_b)\right)V_0 + e^{S_0^+}(\bar{\rho}^+)^{\gamma}(r_b) V_4 + \frac {[\rho U_1 U_2] \ma_2 + [\rho U_1 U_3] \ma_3}{\ma},\\
			&R_{13}(\bm{\Psi}^-(\xi,\th,z) - \bar {\bm{\Psi}}^-(\xi),{\bf V}(\xi,\th,z), V_7 (\th,z))\\
			&= B^-(\xi,\th,z)- {B}_0^-
			-\frac{1}{2}\sum_{j=1}^3 V_j^2(\xi,\th,z)
			-\bar{U}_{1}^+(\xi)V_1+\bar{U}_{1}^+(r_b)V_1 + \frac{(\bar c^+)^2(r_b)}{\bar\rho^+(r_b)} V_0
			\\
			&\quad
			-\frac{\gamma}{\gamma-1} \left(e^{(S_0^+ +V_4(\xi,\th,z))}(\rho(\xi,\th,z))^{\gamma-1} -e^{S_0^+} (\bar \rho^+(\xi))^{\gamma-1}\right)+ \frac{\gamma  e^{S_0^+}(\bar\rho^+(r_b))^{\gamma-1}}{\gamma-1} V_4.
		\end{aligned}
	\end{equation*}
	Note that
	\begin{equation*}
		( \bar\rho^{\pm}  \bar U^{\pm})'(r)= -\frac{1}{r}(\bar{\rho}^{\pm} \bar U^{\pm})(r),\quad\quad
		\left(\bar{\rho}^{\pm} ( \bar U^{\pm})^2+ \bar{P}^{\pm}\right)'(r)=-\frac{1}{r}\left(\bar\rho^{\pm} (\bar U^{\pm})^2\right)(r).
	\end{equation*}
	Thus there hold
	\begin{equation*}
		[\bar{\rho}\bar U](r_b+V_7)=O(V_7^2),\quad [\bar{\rho} \bar U^2+\bar {P}](r_b+V_7)-\frac{1}{r_b}[\bar {P}(r_b)]V_7=O(V_7^2),
	\end{equation*}
	and
	\begin{equation}\label{rh21}
		|R_{1i}| \le  C_0 \left(|{\bm \Psi}^- (\xi ,\th,z) - \bar{\bm \Psi}^-(\xi)| + |{\bf V} (\xi ,\th,z)|^2 + |V_7 (\th,z)|^2\right), \quad i = 1, 2,3.
	\end{equation}
	By solving the algebraic equations in \eqref{rh5}, one gets
	\begin{equation}\label{rh10}
		\begin{cases}
			V_1(\xi,\th,z)=b_1 V_7(\th,z)+R_1 (\bm{\Psi}^-(\xi,\th,z) - \bar {\bm{\Psi}}^-(\xi),{\bf V}(\xi,\th,z), V_7 (\th,z)),\\
			V_4(\xi,\th,z)=b_2 V_7(\th,z)+R_2(\bm{\Psi}^-(\xi,\th,z) - \bar {\bm{\Psi}}^-(\xi),{\bf V}(\xi,\th,z), V_7 (\th,z)),\\
			V_5(\xi,\th,z)=B^-(\xi,\th,z)-{B}_0^-,\\
			V_6(\xi,\th,z)=Y^-(\xi,\th,z),\\
		\end{cases}
	\end{equation}
	where
	\begin{equation*}
		\begin{aligned}
			&b_{1}=\frac{\gamma  \bar U^+(r_b)[\bar{P}(r_b)]}{ r_b\bar\rho^+(r_b) \left((\bar c^+)^2(r_b)-( \bar U^+)^2(r_b)\right)}
			>0,\quad
			b_{2}=\frac{(\gamma-1)[ \bar P(r_b)]}{r_b e^{S_0^+}(\bar\rho^+)^{\gamma}(r_b)}>0 ,\\
			&R_1=\frac{\left((\bar c^+)^2(r_b)+\gamma( \bar U^+)^2(r_b)\right) R_{11}-\gamma  \bar U^+(r_b)R_{12} + (\gamma-1)(\bar \rho^+  \bar U^+)(r_b)R_{13}}{ \bar\rho^+(r_b) \left((\bar c^+)^2(r_b)-( \bar U^+)^2(r_b)\right)}:=\sum_{i=1}^3b_{1i}R_{1i},\\
			&R_2=\frac{\gamma-1}{e^{S_0^+}( \bar\rho^+)^{\gamma}(r_b)}
			\left(\bar U^+(r_b)R_{11}-R_{12}+ \bar\rho^+(r_b) R_{13}\right):=\sum_{i=1}^3b_{2i}R_{1i}.
		\end{aligned}
	\end{equation*}
	\par In the following,  the superscript ``+" in $ \bar\rho^+,\bar U^+,\bar P^+, S_0^+$ will be ignored to simplify the notations. It follows from  the definition of the Bernoulli quantity that
	\begin{equation}\label{2-17}
		\begin{aligned}
			&V_1(r_1,\th,z)+\frac{ e^{S_0}\bar\rho^{\gamma-1}(r_1)V_4(r_1,\th,z)}{(\gamma-1) \bar U(r_1)}
			=\frac{V_5(r_1,\th,z)}{ \bar U(r_1)} - \frac{\sigma P_{ex}(\th,z)}{( \bar\rho \bar U)(r_1)}\\
			& \quad
			-\frac{E({\bf V}(r,\th,z))}{\bar U(r_1)}-\frac{1}{2\bar U(r_1)}\left( V_1^2+V_2^2+V_3^2 \right)(r_1,\th,z), \\
		\end{aligned}
	\end{equation}
	where
	\begin{equation*}
		\begin{aligned}
			E({\bf V}(r,\th,z))=\frac{\gamma}{\gamma-1}e^{\frac{S_0+V_4}{\gamma}}
			P^{\frac{\gamma-1}{\gamma}}-\frac{\gamma}{\gamma-1}
			e^{\frac{S_0}{\gamma}} \bar P^{\frac{\gamma-1}{\gamma}}
			-\frac{1}{ \bar \rho(r)}(P-  \bar P)-\frac{ e^{S_0}\bar \rho^{\gamma-1}(r)}{\gamma-1} V_4.
		\end{aligned}
	\end{equation*}
	Note that $E$ is an error term that can be bounded by
	\begin{equation}\label{rh12}
		| E({\bf V})|\leq C_*|{\bf V}|^2.
	\end{equation}
	The boundary conditions for $V_2$ and $V_3$ on the nozzle walls are
	\begin{equation}\label{slip15}\begin{cases}
			V_2(r,\pm \th_0,z) =0, &    (r, z) \in [r_b+V_7, r_1] \times [-L, L],\\
			V_3(r,\th,\pm L) =0, &  (r, \th) \in [r_b+V_7, r_1] \times [-\th_0, \th_0].
	\end{cases}\end{equation}
	\par Then to solve  the problem \eqref{1-2} with \eqref{1-7-SUP}, \eqref{1-9-SUP}, \eqref{1-12-sub} and \eqref{1-18-sub} is equivalent to find a function $V_7$ defined on $D$ and a vector function $(V_1,\cdots, V_6)$ defined on the $\Omega^+$, which solve the equations \eqref{2-1-a}-\eqref{rho2-1} with boundary conditions \eqref{rh2}, \eqref{rh10}, \eqref{2-17} and \eqref{slip15}.
	\subsection{The coordinates transformation}\noindent
	\par In order to deal with the free boundary value problem, it is convenient to reduce it to a fixed boundary value problem by setting
	\begin{equation}\label{ct1}
		\begin{cases}
			y_1 = \frac {r - \xi }{r_1 - \xi } (r_1 - r_b) + r_b = \frac {r - W_7 - r_b}{r_1 - W_7 - r_b} (r_1 - r_b) + r_b, \\
			y_2 = \th, \quad y_3 = z,
		\end{cases}
	\end{equation}
	where
	\begin{equation*}
		W_7 (\th,z) = \xi (\th,z) - r_b.
	\end{equation*}
	Then
	\begin{equation*}
		\begin{cases}
			\begin{aligned}
				& r = y_1 + \frac {r_1- y_1}{r_1 - r_b}W_7 \eqqcolon   J_0^{W_7}, \\
				&\p_{r} = \frac {r_1 - r_b}{r_1 - W_7- r_b} \p_{y_1}  \eqqcolon    J_1^{W_7}, \\
				&\frac1r\p_{\th} = \frac1{J_0^{W_7}}\left(\p_{y_2} + \frac {(y_1 - r_1) \p_{y_2} W_7}{r_1 - W_7 - r_b} \p_{y_1}\right) \eqqcolon J_2^{W_7}, \\
				&\p_{z }= \p_{y_3} + \frac {(y_1-r_1) \p_{z} W_7}{r_1 - W_7 - r_b} \p_{y_1} \eqqcolon J_3^{W_7},
			\end{aligned}
		\end{cases}
	\end{equation*}
	and the domain $\Omega^+$ becomes
	\begin{equation*}
		\mn = \{ (y_1, y') : y_1 \in (r_b, r_1),\ y' = (y_2, y_3) \in D \}.
	\end{equation*}
	The nozzle walls of the domain $\mn $ are denoted by
	\begin{equation*}
		\begin{aligned}
			&\Sigma_{\th_0}^\pm=\{(y_1, y'):r_b\leq y_1 \leq r_1,\  y_2=\pm \th_0,\ -L \leq y_3 \leq L\},\\
			&\Sigma_{L}^\pm=\{(y_1, y'):r_b\leq y_1 \leq r_1,\  -\th_0 \leq y_2\leq\th_0,\ y_3=\pm L\}.\\
		\end{aligned}
	\end{equation*}
	\par Set
	\begin{equation*}
		\begin{cases}
			W_j (y_1,y') = V_j \left(y_1 + \frac {r_1 - y_1}{r_1 - r_b} W_7, y_2, y_3 \right), \ \ j = 1, \cdots , 6, \quad  {\bf W} = (W_1, \cdots, W_6),\\
			\tilde{\omega}_j (y_1,y') = \omega_j \left(y_1 + \frac {r_1 - y_1}{r_1 - r_b} W_7, y_2, y_3 \right),\ \   j = 1, 2, 3.
		\end{cases}
	\end{equation*}
	Then the density, the pressure and the  temperature  in \eqref{2-11} can be expressed as
	\begin{equation}\label{2-9-1}
		\begin{cases}
			\e\rho(y_1, y^\prime)
			={\left(\frac{\gamma-1}{ \gamma e^{( S_0+W_4)}}\right)}^{\frac{1}{\gamma-1}} \left( B_0^-+W_5 -\frac{1}{2} \left( (W_1+ \bar U(J_0^{W_7}))^2+W_2^2+W_3^2 \right)\right)^{\frac{1}{\gamma-1}},\\
			\e P(y_1, y^\prime)
			={\left(\frac{\gamma-1}{ \gamma  e^{( S_0+W_4)}}\right)}^{\frac{\gamma}{\gamma-1}} \left(B_0^-+ W_5 -\frac{1}{2} \left( (W_1+ \bar U(J_0^{W_7}))^2+W_2^2+W_3^2 \right)\right)^{\frac{\gamma}{\gamma-1}},\\
			\e T(y_1, y^\prime)=\frac1{\gamma}\left( B_0^-+ W_5 -\frac{1}{2} \left( (W_1+ \bar U(J_0^{W_7}))^2+W_2^2+W_3^2 \right)\right).\\
		\end{cases}
	\end{equation}
	\par In the $(y_1,y')$ coordinates, the transonic shock problem can be reformulated as follows. The shock front will be determined by the first equation in \eqref{rh10} as follows
	\begin{equation}\label{rh11-2}
		W_7(y')=\frac1{b_1}W_1(r_b, y')  -\frac1{b_1} R_1({\bf W}(r_b,y'),W_7(y')),
	\end{equation}
	where $$R_1 ({\bf W}(r_b,y'),W_7(y'))=\sum_{i=1}^3b_{1i}R_{1i} ({\bf W}(r_b,y'),W_7(y')).$$
	It follows from \eqref{2-1-a}--\eqref{2-3-c} and \eqref{rh10} that $W_6$, $W_5$ and $W_4$ satisfy
	\begin{equation}\label{2-1-a-a}
		\begin{cases}
			\biggl( J_1^{W_7} + \frac{W_2}{\bar{U} + W_1} J_2^{W_7} + \frac{W_3}{\bar{U} + W_1} J_3^{W_7} \biggr) W_6 = -\frac{\phi(\e T) W_6}{\bar{U} + W_1}, \\
			W_6(r_b, y') = Y^-(r_b + W_7(y'), y'),
		\end{cases}
	\end{equation}
	\begin{equation}\label{2-1-a-b}
		\begin{cases}
			\biggl( J_1^{W_7} + \frac{W_2}{\bar{U} + W_1} J_2^{W_7} + \frac{W_3}{\bar{U} + W_1} J_3^{W_7} \biggr) W_5 = \frac{\phi(\e T) W_6}{\bar{U} + W_1}, \\
			W_5(r_b, y') = B^-(r_b + W_7(y'), y') - B_0^-,
		\end{cases}
	\end{equation}
	and
	\begin{equation}\label{2-1-a-c}
		\begin{cases}
			\biggl( J_1^{W_7} + \frac{W_2}{\bar{U} + W_1} J_2^{W_7} + \frac{W_3}{\bar{U} + W_1} J_3^{W_7} \biggr) W_4 = \frac{\phi(\e T) W_6}{\e T (\bar{U} + W_1)}, \\
			W_4(r_b, y') = b_2 W_7(y') + R_2(\mathbf{W}(r_b, y'), W_7(y')),
		\end{cases}
	\end{equation}
	where
	$$
	R_2\bigl( \mathbf{W}(r_b, y'), W_7(y') \bigr) = \sum_{i=1}^3 b_{2i} R_{1i}\bigl( \mathbf{W}(r_b, y'), W_7(y') \bigr).
	$$
	Furthermore, after the coordinates transformation, the equation \eqref{rh2} is transformed to be
	\begin{equation}\label{rh11}
		\begin{cases}
			\p_{y_2} W_7(y') = r_bb_0 W_2 (r_b, y') + r_bg_2 ({\bf W} (r_b, y'), W_7 (y')),\\
			\p_{y_3} W_7 (y') = b_0 W_3 (r_b, y') + g_3 ({\bf W} (r_b, y'), W_7 (y')),\\
		\end{cases}
	\end{equation}
	where
	\begin{equation}\label{rh11-1}
		\begin{cases}
			g_2 ({\bf W} (r_b, y'), W_7 (y'))=  \frac1{ r_b}\left(\frac {(r_b+W_7)\ma_2  ({\bf W} (r_b, y'), W_7 (y'))}{\ma ({\bf W} (r_b, y'), W_7 (y'))} - r_b b_0 W_2 (r_b, y') \right),\\
			g_3 ({\bf W} (r_b, y'), W_7 (y'))=  \frac {\ma_3  ({\bf W} (r_b, y'), W_7 (y'))}{\ma ({\bf W} (r_b, y'), W_7 (y'))} - b_0 W_3 (r_b, y').
		\end{cases}
	\end{equation}
	Note that the first two equations in \eqref{rh11} are equivalent to
	\begin{equation}\label{rh16}
		\begin{cases}
			F_2 (y') =  \frac1{ r_b}\p_{y_2} W_7(y') - b_0 W_2 (r_b, y') - g_2 ({\bf W} (r_b, y'),W_7(y')),\\
			F_3 (y') = \p_{y_3} W_7 (y') - b_0 W_3 (r_b, y') - g_3({\bf W} (r_b, y'), W_7 (y')).
		\end{cases}
	\end{equation}
	The following lemma  is essential for us to solve the transonic shock problem.
	\begin{lemma}\label{lm1}
		Let $(F_2,F_3)\in \left(C^1(\overline{D})\right)^2$. Then the following
		two statements are equivalent.
		\begin{enumerate}[\rm(i)]
			\item $F_2 = F_3 \equiv 0 $ on $\overline{D}$;
			\item $F_2$ and $F_3$ solve the following problem
			\begin{equation}\label{euler8}
				\begin{cases}
					\p_{y_2} F_3 - r_b\p_{y_3} F_2 = 0,  &{\rm{in }}\  D, \\
					\p_{y_2} F_2 + r_b\p_{y_3} F_3 = 0,  &{\rm{in }} \ D, \\
					F_2 (\pm \th_0 , y_3)  = 0,    &y_3 \in [-L,L], \\
					F_3 (y_2, \pm L ) = 0,  &y_2 \in [-\th_0,\th_0].
				\end{cases}
			\end{equation}
		\end{enumerate}
	\end{lemma}
	\par The first equation in \eqref{euler8} implies
	\begin{equation}\label{rh18}
		\frac1{r_b}\p_{y_2} W_3 (r_b, y') - \p_{y_3} W_2 (r_b, y') = \frac 1{b_0} \left( \p_{y_3} g_2 ({\bf W} (r_b, y'), W_7 (y')) - \frac1{r_b}\p_{y_2} g_3 ({\bf W} (r_b, y'), W_7 (y')) \right),
	\end{equation}
	which gives the boundary data on the shock front for the first component of the
	vorticity.
	The second equation in \eqref{euler8} yields
	\begin{equation}\label{rh19}
		\begin{aligned}
			&\frac1{r_b^2}\p_{y_2}^2 W_7 (y') + \p_{y_3}^2 W_7 (y') - \frac{b_0}{r_b} \p_{y_2} W_2 (r_b, y') - b_0 \p_{y_3} W_3 (r_b, y') \\
			& \quad = \frac{1}{r_b}\p_{y_2}g_2 ({\bf W} (r_b, y'), W_7 (y'))+ \p_{y_3}g_3 ({\bf W} (r_b, y'), W_7 (y')).
		\end{aligned}
	\end{equation}
	This, together with \eqref{rh11-2}, gives that
	\be\label{shock19}
	\frac{1}{r_b^2}\p_{y_2}^2 W_1 (r_b, y')+\p_{y_3}^2 W_1 (r_b, y') -b_0b_1\left(\frac{1}{r_b}\p_{y_2} W_2+\p_{y_3} W_3\right)(r_b,y')=q_1({\bf W}(r_b,y'),W_7 (y')),
	\ee
	with
	\be\no
	q_1=b_1\left(\frac{1}{r_b}\p_{y_2}g_2+\p_{y_3}g_3\right)({\bf W}(r_b,y'), W_7 (y'))+\left(\frac{1}{r_b^2}\p_{y_2}^2 R_1+\p_{y_3}^2 R_1\right)({\bf W}(r_b,y'), W_7 (y')).
	\ee
	The condition \eqref{shock19} can be used as the boundary condition for the deformation-curl system associated with the velocity field on the shock front. Furthermore, one can express the boundary conditions in \eqref{euler8} as
	\begin{align}\label{bc4}
		& \left( \frac{1}{r_b} \partial_{y_2} W_1 - b_0 b_1 W_2 \right)(r_b, \pm \theta_0, y_3) = q_2^\pm\bigl( \mathbf{W}(r_b, \pm \theta_0, y_3), W_7(\pm \theta_0, y_3) \bigr), \quad y_3 \in [-L, L], \\ \label{bc17}
		& \left( \partial_{y_3} W_1 - b_0 b_1 W_3 \right)(r_b, y_2, \pm L) = q_3^\pm\bigl( \mathbf{W}(r_b, y_2, \pm L), W_7(y_2, \pm L) \bigr), \quad y_2 \in [-\theta_0, \theta_0],
	\end{align}
	with
	\begin{equation*}
		\begin{aligned}
			&q_2^{\pm} ({\bf W} (r_b, \pm \th_0, y_3 ), W_7(\pm \th_0 , y_3))
			=\left( \frac1{r_b}\p_{y_2} R_1 ({\bf W} (r_b, \cdot ),  W_7 (\cdot))  + b_1 g_2 ({\bf W} (r_b, \cdot ),  W_7 (\cdot)) \right)(\pm \th_0 , y_3), \\
			&q_3^{\pm}({\bf W} (r_b,  y_2, \pm L ), W_7 (y_2, \pm L))
			= \left(\p_{y_3} R_1 ({\bf W} (r_b, \cdot ), W_7 (\cdot)) + b_1 g_3 ({\bf W} (r_b, \cdot ), W_7 (\cdot)) \right)(y_2, \pm L).
		\end{aligned}
	\end{equation*}
	The roles of \eqref{bc4} and \eqref{bc17} will be indicated later.
	\par The vorticity can be determined as follows. \eqref{om3-1} is changed to be
	\begin{equation}\label{om7}
		\begin{aligned}
			& \bigg( J_1^{W_7} + \f {W_2}{\bar{U}  + W_1} J_2^{W_7} + \f {W_3}{\bar{U} + W_1} J_3^{W_7} \bigg) \e{\om}_1
			+ \bigg(\f1{J_0^{W_7}}+ J^{W_7}_2 \bigg( \f {W_2}{\bar{U}(J_0^{W_7})  + W_1}\bigg) \\
			& \quad + J^{W_7}_3 \bigg( \f {W_3}{\bar{U}(J_0^{W_7}) + W_1} \bigg) \bigg) \e\om_1
			= H_1 ({\bf W}, W_7),
		\end{aligned}
	\end{equation}
	with
	\begin{equation*}
		\begin{aligned}
			H_1 ({\bf W}, W_7)& = J^{W_7}_3 \left( \f 1{\bar{U}  + W_1} \right) J^{W_7}_2 W_5
			-  J^{W_7}_2 \left(\f 1{\bar{U}  + W_1} \right) J^{W_7}_3 W_5\\
			&\quad +  J^{W_7}_2 \left(\frac{e^{( S_0^++W_4)}\e\rho^{\gamma-1}}{(\gamma-1)(\bar U^++ W_1)}\right) J^{W_7}_3 W_4
			- J^{W_7}_3 \left(\frac{e^{( S_0^++W_4)}\e\rho^{\gamma-1}}{(\gamma-1)(\bar U^++ W_1)}\right)J^{W_7}_2W_4  .
		\end{aligned}
	\end{equation*}
	The boundary data for $\tilde{\om}_1$ is given by \eqref{rh18} at $y_1 = r_b$,
	\begin{equation}\label{bc6}
		\tilde{\om}_1 (r_b, y') =\frac 1{b_0} \left( \p_{y_3} g_2 ({\bf W} (r_b, y'), W_7 (y')) - \frac1{r_b}\p_{y_2} g_3 ({\bf W} (r_b, y'), W_7 (y')) \right) + g_4 ({\bf W} (r_b, y'), W_7 (y')) ,
	\end{equation}
	where
	\begin{equation}\label{bc6-bo}
		\begin{aligned}
			g_4 ({\bf W}({\bf y}) , W_7 (y'))
			= &
			\frac{(y_1-r_1) W_7\p_{y_2}W_3(r_b,y')}{y_1((r_1-r_b)y_1+(r_1-y_1)W_7)}-\frac{(y_1-r_1)
				\p_{y_3}W_7\p_{y_1}W_2(r_b,y')}{r_1-r_b-W_7}\\
			& +\frac{r_1-r_b}{(r_1-r_b)y_1+(r_1-y_1)W_7}\frac{(y_1-r_1)\p_{y_2}W_7 \p_{y_1}W_3(r_b,y')}{r_1-r_b-W_7}.
		\end{aligned}
	\end{equation}
	Next, it follows from \eqref{om2-1} that
	\begin{equation*}
		\begin{cases}
			\e{\om}_2 = J_3^{W_7 } W_1 - J_1^{W_7} W_3  =
			\frac {W_2 \e\omega_1 + J_3^{W_7}W_5}{\bar U(J_0^{W_7})+W_1}-\frac{e^{( S_0+W_4)}\e\rho^{\gamma-1}}{(\gamma-1)(\bar U(J_0^{W_7})+W_1)} J_3^{W_7}W_4
			, \\
			\e{\om}_3 = J_1^{W_7} W_2 - J_2^{W_7} W_1+\frac{W_2}{J_0^{W_7}}  =
			\frac {W_3 \e\omega_1 - J_2^{W_7}W_5}{\bar U(J_0^{W_7})+W_1}+\frac{e^{( S_0+W_4)}\e\rho^{\gamma-1}}{(\gamma-1)(\bar U(J_0^{W_7})+W_1)} J_2^{W_7}W_4.
		\end{cases}
	\end{equation*}
	Then one gets
	\be\label{euler9}
	&& \frac1{y_1}\p_{y_2} W_3 - \p_{y_3} W_2 = \e{\om}_1 + H_2 ({\bf W}, W_7), \\\label{euler10}
	&& \p_{y_3} W_1 - \p_{y_1} W_3 +\f  {e^{S_0}\bar \rho^{\gamma-1}(y_1)\p_{y_3} W_4}{(\gamma-1)\bar{U} (y_1)}
	=
	\f {W_2 \e{\om}_1 +J_3^{W_7}W_5}{\bar{U}(J_0^{W_7}) + W_1} + H_3 ({\bf W}, W_7) , \\\label{euler11}
	&& \p_{y_1} W_2+\frac{W_2}{y_1} - \frac{1}{y_1}\p_{y_2} W_1 - \f  {e^{S_0}\bar \rho^{\gamma-1}(y_1)\p_{y_2} W_4}{(\gamma-1)y_1\bar{U} (y_1) }
	= \f {W_3 \e{\om}_1 -J_2^{W_7}W_5}{\bar{U}(J_0^{W_7})  + W_1} + H_4 ({\bf W}, W_7),
	\ee
	where
	\begin{equation*}
		\begin{aligned}
			H_2({\bf W}, W_7) & =
			\frac{(y_1-r_1)\p_{y_3}W_7 \p_{y_1}W_2}{(r_1-r_b-W_7)}-\frac{(y_1-r_1)W_7\p_{y_2} W_3}{y_1((r_1-r_b)y_1+(r_1-y_1)W_7)}\\
			&\quad-\frac{r_1-r_b}{(r_1-r_b)y_1+(r_1-y_1)W_7}\frac{(y_1-r_1)\p_{y_2}W_7 \p_{y_1}W_3}{r_1-r_b-W_7},\\
			H_3 ({\bf W}, W_7) & =
			-\frac{(y_1-r_1)\p_{y_3} W_7 \p_{y_1} W_1}{r_1-r_b-W_7} + \frac{W_7\p_{y_1}W_3}{r_1-r_b-W_7}
			-\f {e^{( S_0+W_4)}\e\rho^{\gamma-1}(J_0^{W_7}) J^{W_7}_3 W_4}{(\gamma-1) \left(\bar{U} (J_0^{W_7})+ W_1 \right)}
			+\f  {e^{ S_0}\bar \rho^{\gamma-1}(y_1)\p_{y_3} W_4}{(\gamma-1)\bar{U}(y_1) }, \\
			H_4 ({\bf W}, W_7) & =
			-\frac{W_7 \p_{y_1} W_2}{r_1-r_b-W_7}+ \frac{(r_1-y_1)W_7(W_2+\p_{y_2} W_1)}{y_1((r_1-r_b)y_1+(r_2-y_1)W_7)}+
			\frac{1}{J_0^{W_7}}\frac{(y_1-r_1)\p_{y_2}W_7\p_{y_1}W_1}{r_1-r_b-W_7}\\
			&\quad +\f {e^{( S_0+W_4)}\e\rho^{\gamma-1}(J_0^{W_7}) J^{W_7}_2 W_4}{(\gamma-1)\left(\bar{U} (J_0^{W_7})+ W_1\right)}
			-\f  {e^{ S_0}\bar \rho^{\gamma-1}(y_1)\p_{y_2} W_4}{(\gamma-1)y_1 \bar{U}(y_1) }.
		\end{aligned}
	\end{equation*}
	Furthermore, the equation \eqref{rho2-1} can be rewritten as
	\begin{equation}\label{rho2-1-co}
		\begin{aligned}
			d_1(y_1)\p_{y_1} W_1 + \frac1{y_1}\p_{y_2} W_2 +   \p_{y_3} W_3+\frac{W_1}{y_1}+d_2(y_1)W_1=d_3(y_1)W_5+d_4(y_1)W_6
			+H_0({\bf W},W_7),
		\end{aligned}
	\end{equation}
	where
	\begin{align*}
		&H_0({\bf W},W_7)=-\frac{(\gamma-1)\left(J^{W_7}_1 W_1+\frac{W_1}{J^{W_7}_0}\right)}{\bar c^2(J^{W_7}_0)} W_5+ \frac{\bar{U}'(J^{W_7}_0)+J^{W_7}_1 W_1}{\bar c^2}\left(\frac{\gamma+1}{2} W_1^2+\frac{\gamma-1}{2}(W_2^2+W_3^2)\right)\\
		&\quad+\frac{(\gamma-1)(\bar{U}(J^{W_7}_0)+W_1)}{2J^{W_7}_0 \bar c^2(J^{W_7}_0)}\sum_{j=1}^3 W_j^2+\frac{(\gamma+1)\bar{U}(J^{W_7}_0)W_1 J^{W_7}_1 W_1+ (\gamma-1)\bar{U}(J^{W_7}_0) \frac{W_1^2}{J^{W_7}_0}}{\bar c^2(J^{W_7}_0)}\\
		&\quad- \frac{1}{\bar c^2(J^{W_7}_0)}\bigg((\gamma-1)W_5-\frac{\gamma-1}{2}\sum_{j=1}^3 W_j^2-(\gamma-1)\bar{U}(J^{W_7}_0)W_1\bigg)\bigg(J^{W_7}_2 W_2+J^{W_7}_3 W_3\bigg)\\
		&\quad
		+\frac{1}{\bar c^2(J^{W_7}_0)}\left(W_2^2J^{W_7}_2 W_2 +W_3^2 J^{W_7}_3W_3\right)+ \frac{\bar{U}(J^{W_7}_0)+W_1}{\bar{c}^2(J^{W_7}_0)}(W_2J^{W_7}_1 W_2+W_3J^{W_7}_1W_3)\\
		&\quad+\frac{W_2}{(\bar c^2(J^{W_7}_0)}\left((\bar{U}+W_1)J^{W_7}_2 W_1 +{W_3}J^{W_7}_2W_3\right)+ \frac{W_3}{\bar{c}^2(J^{W_7}_0)}\left((\bar{U}(J^{W_7}_0)+W_1)J^{W_7}_3 W_1 +W_2J^{W_7}_3 W_2\right)\\
		&\quad+\frac{1}{\bar c^2(J^{W_7}_0)}\left((\gamma-1)\phi( T)-(\gamma-1)\phi(\bar T^+)(J^{W_7}_0)\right)W_6 +d_1(y_1)\p_{y_1} W_1+ \frac{W_1}{y_1} \\
		&\quad -\bigg(d_1(J^{W_7}_0)J^{W_7}_1 W_1+\frac{W_1}{J^{W_7}_0} + J^{W_7}_2 W_2 +   J^{W_7}_3W_3+d_2(J^{W_7}_0)W_1\bigg)+ \frac1{y_1}\p_{y_2} W_2 +   \p_{y_3} W_3+d_2(y_1)W_1\\
		&\quad+\left(d_3(J^{W_7}_0)-d_3(y_1)\right)W_5
		+\left(d_4(J^{W_7}_0)-d_4(y_1)\right)W_6.\\
	\end{align*}
	The  boundary condition at the exit becomes
	\begin{equation}\label{2-17-ex}
		\begin{aligned}
			&W_1(r_1,y')+\frac{ e^{S_0}\bar\rho^{\gamma-1}(r_1)W_4(r_1,y')}{(\gamma-1) \bar U(r_1)}
			=\frac{W_5(r_1,y')}{ \bar U(r_1)} - \frac{\sigma P_{ex}(y')}{( \bar\rho \bar U)(r_1)}\\
			& \quad
			-\frac{E({\bf W}(r_1,y'))}{\bar U(r_1)}-\frac{1}{2\bar U(r_1)}\bigg ( W_1^2+W_2^2+W_3^2\bigg )(r_1,y'), \\
		\end{aligned}
	\end{equation}
	where
	\begin{equation*}
		\begin{aligned}
			E({\bf W}(r_1,y'))=\frac{\gamma e^{\frac{S_0+W_4(r_1,y')}{\gamma}}
				P^{\frac{\gamma-1}{\gamma}}}{\gamma-1}-\frac{\gamma e^{\frac{S_0}{\gamma}} \bar P^{\frac{\gamma-1}{\gamma}}}{\gamma-1}
			-\frac{P(r_1,y')-  \bar P(r_1)}{ \bar \rho(r_1)}-\frac{ e^{S_0}\bar \rho^{\gamma-1}(r_1)}{\gamma-1} W_4(r_1,y').
		\end{aligned}
	\end{equation*}
	The boundary conditions  on the nozzle walls are
	\begin{equation}\label{slip16}\begin{cases}
			W_2(y_1,\pm \th_0,y_3) =0, &    (y_1, y_3) \in \Sigma_{\th_0}^\pm,\\
			W_3(y_1,y_2,\pm L) =0, &   (y_1, y_2) \in \Sigma_{L}^\pm.
	\end{cases}\end{equation}
	\par Therefore after the coordinates transformation \eqref{ct1}, the transonic shock problem \eqref{1-2} with \eqref{1-7-SUP}, \eqref{1-9-SUP}, \eqref{1-12-sub} and \eqref{1-18-sub} is equivalent to  the following problem:
	\begin{problem}\label{probl2} Find a function $W_7$ defined on $D$ and a vector function $(W_1, \cdots, W_6)$ defined on $\mn$, which solve the equations \eqref{2-1-a-a}-\eqref{2-1-a-c}, \eqref{om7} and \eqref{euler9}-\eqref{rho2-1-co} with the boundary conditions \eqref{rh11}, \eqref{shock19}-\eqref{bc17},  \eqref{bc6}, \eqref{2-17-ex} and \eqref{slip16}.
	\end{problem}
	\begin{theorem}\label{main-1}
		Assume that the compatibility conditions \eqref{1-8-SUP} and \eqref{1-13-sub}  hold. There exists a suitable constant $\sigma_0^\ast > 0$ depending only on the background solution  and the boundary data such that if $0 < \sigma < \sigma_0^\ast$, the problem \eqref{2-1-a-a}-\eqref{2-1-a-c}, \eqref{om7} and \eqref{euler9}-\eqref{rho2-1-co} with the boundary conditions \eqref{rh11},  \eqref{bc6}, \eqref{2-17-ex} and \eqref{slip16} has a unique solution $  (W_1, W_2, W_3, W_4,W_5,W_6)(y_1,y')  $ with the shock front $\ms: y_1 =  W_7(y')$ satisfying the following properties.
		\begin{enumerate}[\rm(1)]
			\item The function $W_7 \in C^{3, \alpha} (\overline{D})$ satisfies the estimate
			\begin{equation}\label{est-xi-1}
				\| W_7\|_{C^{3, \alpha} (\overline{D})} \le C_0^* \sigma,
			\end{equation}
			and the compatibility conditions
			\begin{equation}\label{bc-xi-1}
				\begin{cases}
					\p_{y_2} W_7 (\pm \th_0, y_3)  = \p_{y_2}^3W_7 (\pm \th_0, y_3)  = 0, &   y_3 \in [-L, L], \\
					\p_{y_3} W_7 (y_2, \pm L)  = \p_{y_3}^3 W_7 (y_2, \pm L) = 0, &   y_2 \in [-\th_0, \th_0].
				\end{cases}
			\end{equation}
			\item The solution $(W_1, W_2, W_3, W_4,W_5,W_6) \in \left(C^{2, \alpha} (\overline{\mn})\right)^6$ satisfies
			\begin{equation}\label{est-solu-1}
				\sum_{i=1}^6\| W_i \|_{C^{2, \alpha } (\overline{\mn})} \le C_0^* \sigma,
			\end{equation}
			with the compatibility conditions
			\begin{equation}\label{bc-solu-1}
				\begin{cases}
					(W_2, \p_{y_2}^2 W_2) (y_1, \pm \th_0, y_3) =
					\p_{y_2} (W_1, W_3, W_4,
					W_5,W_6) (y_1, \pm \th_0, y_3) = 0, &(y_1,y_3)\in\Sigma_{\th_0}^\pm, \\
					(W_3, \p_{y_3}^2 W_3) (y_1, y_2, \pm L)=
					\p_{y_3} (W_1, W_2, W_4,
					W_5,W_6) (y_1, y_2, \pm L) = 0,  &(y_1,y_2)\in\Sigma_{L}^\pm.
				\end{cases}
			\end{equation}
		\end{enumerate}
		Here  $C_0^* $ is a positive constant  depending only on the background solution  and the boundary data.
	\end{theorem}
	\section{The stability analysis}\label{iteration}\noindent
	\par In this section, we will design an iteration to prove Theorem \ref{main-1}. Define the solution class $\mathfrak{W}$ consists of the vector function $(W_1, \cdots,W_6, W_7) \in \left(C^{2, \al} (\overline\mn)\right)^6 \times C^{3, \al} (\overline{D})$ satisfying the estimate
	\begin{equation}\label{est3}
		\| ({\bf W}, W_7) \|_{\mathfrak{W}} \co \sum_{i = 1}^6\| W_i \|_{C^{2, \al} (\overline\mn)} + \| W_7\|_{C^{3, \al} (\overline{D})} \le \delta_0,
	\end{equation}
	and the  compatibility conditions
	\begin{equation}\label{bc20}
		\begin{cases}
			(W_2, \p_{y_2}^2 W_2) (y_1, \pm \th_0, y_3) =
			\p_{y_2} (W_1, W_3, W_4,
			W_5,W_6) (y_1, \pm \th_0, y_3) = 0, &(y_1,y_3)\in \Sigma_{\th_0}^\pm, \\
			(W_3, \p_{y_3}^2 W_3) (y_1, y_2, \pm L)=
			\p_{y_3}(W_1, W_2, W_4,
			W_5,W_6) (y_1, y_2, \pm L) = 0, & (y_1,y_2)\in \Sigma_{L}^\pm,\\
			(\p_{y_2} W_7, \p_{y_2}^3 W_7) (\pm \th_0, y_3) = 0, & y_3 \in [-L , L], \\
			(\p_{y_3} W_7, \p_{y_3}^3 W_7) (y_2, \pm L) = 0, &  y_2 \in [-\th_0, \th_0],
		\end{cases}
	\end{equation}
	where $\delta_0$  is a  positive constant to be determined later.
	For any given $(\h{\bf W}, \h W_7)\in \mathfrak{W}$, we will construct an iterative procedure that generates a new $({\bf W},  W_7)$, and thus one can define a mapping  from $\mathfrak{W}$ to itself by choosing a suitably small positive constant $\delta$. Note that $W_6$, $W_5$ and $W_4$ satisfy the transport equations.  We first solve the transport equation for $W_6$ to obtain its expression.   Using
	the expression of $W_6$,  one can   solve the  transport equations for  $W_5$ and $W_4$  to derive their  representation formulas.  Note that the boundary data for $W_6$, $W_5$ and $W_4$ involves the shock position $W_7$, using the condition \eqref{rh11-2},  $W_6$ and $W_5$ can be regarded as  higher order terms and the principal term in $W_4$ is given by a scalar multiple of $W_1(r_b,y')$. Substituting these into \eqref{euler9}-\eqref{rho2-1-co},  we derive a  first order deformation-curl system for $(W_1,W_2, W_3)$  with a nonlocal term $W_1(r_b,y')$  and some boundary conditions from which $(W_1,W_2, W_3)$ can be solved uniquely. Now we give a detailed derivation of this procedure.
	\par {\bf Step 1. The determination of  $W_7$.}
	\par The shock front $W_7$ is uniquely determined by
	\begin{equation}\label{6-2}
		W_7(y')=\frac{1}{b_{1} }W_1(r_b,y')-\frac{1}{b_{1} }  R_1(\h{\bf W}(r_b,y'),\h W_7(y')),
	\end{equation}
	provided that $W_1(r_b,y')$ is obtained.
	\par {\bf Step 2. Solving the transport equation for $W_6$.}
	\par It follows from \eqref{2-1-a-a} that the fraction of
	unburned gas will be determined by
	\begin{equation}\label{6-3}
		\begin{cases}
			\left(\p_{y_1} + I_2(\h{\bf W},\h W_7)\p_{y_2} + I_3(\h{\bf W},\h W_7)\p_{y_3}\right)W_6= I_4(\h{\bf W},\h W_7) W_6,\\
			W_6(r_b,y')=Y^-(r_b+\h W_7(y'),y'),
		\end{cases}
	\end{equation}
	where
	\begin{equation*}
		\begin{aligned}
			I_2(\h{\bf W},\h W_7)&=\frac{(r_1-r_b-\h W_7)\hat{W}_2}
			{(r_1-r_b)J_0^{\hat{W}_7}
				(\bar{U}(J_0^{\hat{W}_7})+\hat{W}_1)
				+(y_1-r_1)(\hat{W}_2\p_{y_2}\hat{W}_7
				+J_0^{\hat{W}_7}\hat{W}_3\p_{y_3}\hat{W}_7)},\\
			I_3(\h{\bf W},\h W_7)&=\frac{(r_1-r_b-\h W_7)J_0^{\hat{W}_7}\hat{W}_3}
			{(r_1-r_b){ J_0^{\hat{W}_7}}
				(\bar{U}(J_0^{\hat{W}_7})+\hat{W}_1)
				+(y_1-r_1)(\hat{W}_2\p_{y_2}\hat{W}_7
				+J_0^{\hat{W}_7}\hat{W}_3\p_{y_3}\hat{W}_7)},\\
			I_4(\h{\bf W},\h W_7)&=-\frac{(r_1-r_b-\h W_7){ J_0^{\hat{W}_7}}\phi(\h T)}
			{(r_1-r_b){J_0^{\hat{W}_7}}
				(\bar{U}(J_0^{\hat{W}_7})+\hat{W}_1)
				+(y_1-r_1)(\hat{W}_2\p_{y_2}\hat{W}_7
				+J_0^{\hat{W}_7}\hat{W}_3\p_{y_3}\hat{W}_7)},
		\end{aligned}
	\end{equation*}
	with
	\begin{equation*}
		\h T=\frac1{\gamma}\left( B_0^-+\h W_5 -\frac{1}{2} \left( \bar{U}(T_0^{\h W_7})+\h W_1 \right)^2-\frac{1}{2}\h W_2^2-\frac{1}{2} \h W_3^2 \right).
	\end{equation*}
	For any  $({\bf \hat{W}}, \hat{W}_7) \in\mathfrak{W}$, one has $(I_2, I_3,I_4) \in \left(C^{2, \al } (\overline{\mn})\right)^3$.  We consider the following ODE system
	\be\label{ode1}
	\begin{cases}
		\f {\de\h{y}_2 (\tau; y_1,y')}{d \tau} = I_2 (\tau, \h{y}_2 (\tau; y_1,y'), \h{y}_3 (\tau; y_1,y') ), &  \tau \in [r_b, r_1], \\
		\f {\de \h{y}_3 (\tau; y_1,y')}{d \tau} = I_3 (\tau, \h{y}_2 (\tau; y_1,y'), \h{y}_3 (\tau; y_1,y') ), &  \tau \in [r_b, r_1], \\
		\h{y}_2 (y_1; y_1,y') = y_2, \quad  \h{y}_3 (y_1; y_1,y') = y_3.
	\end{cases}
	\ee
	Denote $(\eta_2 (y_1,y'), \eta_3 (y_1,y')) = (\h{y}_2 (r_b; y_1,y'), \h{y}_3 (r_b; y_1,y'))$. Since $({\bf \hat{W}}, \hat{W}_7) \in \mathfrak{W}$, it follows from the compatibility condition \eqref{bc20} that
	\begin{equation}\label{bc21}
		\begin{cases}
			I_2 (y_1, \pm \th_0, y_3) = \p_{y_2} I_3 (y_1, \pm \th_0, y_3) =\p_{y_2}  I_4 (y_1,  \pm \th_0,y_3)= 0, &(y_1,y_3)\in\Sigma_{\th_0}^\pm, \\
			I_3 (y_1, y_2, \pm L) = \p_{y_3} I_2 (y_1, y_2, \pm L)=\p_{y_3} I_4 (y_1, y_2, \pm L) = 0, & (y_1,y_2)\in\Sigma_{L}^\pm.
		\end{cases}
	\end{equation}
	Due to the uniqueness of the solution \eqref{ode1} and \eqref{bc21}, one obtains
	\be\label{bc22}
	\begin{cases}
		\h{y}_2 (\tau; y_1, \pm \th_0, y_3) = \pm \th_0, &  \tau \in [r_b,r_1], \, (y_1,y_3)\in\Sigma_{\th_0}^\pm, \\
		\h{y}_3 (\tau; y_1, y_2, \pm L) = \pm L, &  \tau \in [r_b,r_1], \, (y_1,y_2)\in\Sigma_{L}^\pm,
	\end{cases}
	\ee
	and
	\be\label{bc23}
	\begin{cases}
		\eta_2 (y_1, \pm \th_0, y_3) = \pm \th_0, \  &(y_1,y_3)\in\Sigma_{\th_0}^\pm, \\
		\eta_3 (y_1, y_2, \pm L) = \pm L, \ \ &(y_1, y_2) \in \Sigma_{L}^\pm.
	\end{cases}
	\ee
	The standard theory of systems of ordinary differential equations and \eqref{bc22} guarantee the existence and uniqueness of $(\h{y}_2 (\tau;y_1,y'), \h{y}_3 (\tau; y_1,y'))$ on the whole interval $[r_b, r_1]$. Furthermore, there holds
	\be\label{est5}
	\sum_{i = 2}^3 \| \eta_i (y_1,y') - y_i \|_{C^{2, \al } (\overline\mn)} \le C \| (\h{\bf W}, \h W_7) \|_{\mathfrak{W}}.
	\ee
	\par Differentiating the first equation in \eqref{ode1} with respect to $y_3$ and restricting the resulting equation on $y_3 = \pm L$, one has
	\be\no
	\begin{cases}
		\f {\de}{\de \tau} \p_{y_3} \h{y}_2 (\tau; y_1, y_2, \pm L) = \p_{y_2} I_2 (\tau, \h{y}_2 (\tau; y_1,y'), \h{y}_3 (\tau; y_1,y')) \p_{y_3} \h{y}_2 (\tau; y_1, y_2, \pm L),  \\
		\p_{y_3} \h{y}_2 (y_1; y_1, y_2, \pm L) = 0,
	\end{cases}
	\ee
	from which one obtains that $\p_{y_3} \h{y}_2 (\tau; y_1, y_2, \pm L) = 0$ for any $\tau \in [r_b, r_1]$. Similarly, we have $ \p_{y_2} \h{y}_3 (\tau; y_1,$\\$ \pm \th_0, y_3) = 0$ for any $\tau \in [r_b, r_1]$. Therefore,
	\be\label{bc24}
	\begin{cases}
		\p_{y_3} \eta_2 (y_1, y_2, \pm L) = 0 , & (y_1, y_2) \in \Sigma_{L}^\pm, \\
		\p_{y_2} \eta_3 (y_1, \pm \th_0, y_3) = 0 , & (y_1, y_3) \in \Sigma_{\th_0}^\pm.
	\end{cases}
	\ee
	By the characteristic method, one has
	\begin{equation}\label{6-4}
		W_6 (y_1,y')  =   Y^-(r_b+\h W_7(\eta_2 (y_1,y'), \eta_3 (y_1,y')),\eta_2 (y_1,y'), \eta_3 (y_1,y')) e^{\int_{r_b}^{y_1} I_4(\h{\bf W},\h W_7)(\tau,\h{{y }}_2(\tau;y_1,y'),\h{{ y}}_3(\tau;y_1,y'))\de \tau}.
	\end{equation}
	Then it holds that
	\begin{equation}\label{W6-est}
		\|W_6\|_{C^{2,\alpha}(\overline\mn)}
		\leq  C_*(\sigma+\sigma \delta_0).
	\end{equation}
	It follows from \eqref{1-11-SUP}, \eqref{bc20}, \eqref{bc21} and \eqref{bc24} that
	\be\label{W6-cp}
	\begin{cases}
		\p_{y_2} W_6 (y_1,\pm\th_0, y_3) = 0 , & (y_1, y_3) \in \Sigma_{\th_0}^\pm, \\
		\p_{y_3}  W_6 (y_1,y_2, \pm L) = 0 , & (y_1, y_2) \in \Sigma_{L}^\pm.
	\end{cases}
	\ee
	\par {\bf Step 3. Solving the transport equations for $W_5$ and $W_4$.}
	\par It follows from \eqref{2-1-a-b} and \eqref{2-1-a-c} that the Bernoulli quantity and the entropy will be determined by
	\begin{equation}\label{6-7}
		\begin{cases}
			\left(\p_{y_1} + I_2(\h{\bf W},\h W_7)\p_{y_2} + I_3(\h{\bf W},\h W_7)\p_{y_3}\right)W_5=I_5(W_6,\h{\bf W},\h W_7) ,\\
			W_5(r_b,y')=B^-(r_b+\h W_7(y'),y')- B_0^-,
		\end{cases}
	\end{equation}
	and
	\begin{equation}\label{6-8}
		\begin{cases}
			\left(\p_{y_1} + I_2(\h{\bf W},\h W_7)\p_{y_2} + I_3(\h{\bf W},\h W_7)\p_{y_3}\right)W_4=I_6(W_6,\h{\bf W},\h W_7), \\
			W_4(r_b,y')=b_2 W_7(y')+R_2(\h{\bf W}(r_b,y'),\h W_7(y')),
		\end{cases}
	\end{equation}
	where
	\begin{equation*}
		\begin{aligned}
			I_5(W_6,\h{\bf W},\h W_7)&=\frac{(r_1-r_b-\h W_7)J_0^{\hat{W}_7}\phi(\h T)W_6}
			{(r_1-r_b)J_0^{\hat{W}_7}
				(\bar{U}(J_0^{\hat{W}_7})+\hat{W}_1)
				+(y_1-r_1)(\hat{W}_2\p_{y_2}\hat{W}_7
				+J_0^{\hat{W}_7}\hat{W}_3\p_{y_3}\hat{W}_7)},\\
			I_6(W_6,\h{\bf W},\h W_7)&=\frac{(r_1-r_b-\h W_7)J_0^{\hat{W}_7}\phi(\h T)W_6}
			{\h T\left((r_1-r_b)J_0^{\hat{W}_7}
				\left(\bar{U}(J_0^{\hat{W}_7})+\hat{W}_1 \right)
				+(y_1-r_1) \left(\hat{W}_2\p_{y_2}\hat{W}_7
				+J_0^{\hat{W}_7}\hat{W}_3\p_{y_3}\hat{W}_7 \right) \right)}.\\
		\end{aligned}
	\end{equation*}
	Then  it is easy to verify that $(I_5,I_6) \in \left(C^{2, \al } (\overline{\mn})\right)^2$ satisfies
	\be\label{bc21-hy}
	\begin{cases}
		\p_{y_2} (I_5,I_6) (y_1, \pm \th_0, y_3) = 0, &(y_1,y_3)\in\Sigma_{\th_0}^\pm, \\
		\p_{y_3} (I_5,I_6) (y_1, y_2, \pm L)=0 & (y_1,y_2)\in\Sigma_{L}^\pm.
	\end{cases}
	\ee
	\par Integrating the equation in \eqref{6-7} and \eqref{6-8}  along the trajectory $(\tau,\h{{ y}}_2(\tau;y_1,y'),\h{{y}}_3(\tau;y_1,y'))$ yields that
	\begin{equation}\label{6-9}
		\begin{aligned}
			W_5 (y_1,y')  &=   B^-(r_b+\h W_7(\eta_2 (y_1,y'), \eta_3 (y_1,y')),\eta_2 (y_1,y'), \eta_3 (y_1,y'))-B_0^-\\
			&\quad+\int_{r_b}^{y_1} I_5(W_6,\h{\bf W},\h W_7)\left(\tau,\h{{y }}_2(\tau;y_1,y'),\h{{ y}}_3(\tau;y_1,y')\right)\de \tau,\\
		\end{aligned}
	\end{equation}
	and
	\begin{equation*}
		\begin{aligned}
			W_4 (y_1,y')  &= b_2 W_7(\eta_2 (y_1,y'), \eta_3 (y_1,y'))\\
			&\quad+R_2\left(\h{\bf W}(r_b,\eta_2 (y_1,y'), \eta_3 (y_1,y')), \h W_7(\eta_2 (y_1,y'), \eta_3 (y_1,y'))\right) \\
			&\quad+\int_{r_b}^{y_1} I_6(W_6,\h{\bf W},\h W_7)\left(\tau,\h{{y }}_2(\tau;y_1,y'),\h{{ y}}_3(\tau;y_1,y')\right)\de \tau\\
			&=b_2W_7(y')+b_2\left( W_7(\eta_2 (y_1,y'), \eta_3 (y_1,y'))-W_7(y')\right)\\
			&\quad+R_2\left(\h{\bf W}(r_b,\eta_2 (y_1,y'), \eta_3 (y_1,y')),\h W_7(\eta_2 (y_1,y'), \eta_3 (y_1,y'))\right) \\
			&\quad+\int_{r_b}^{y_1} I_6(W_6,\h{\bf W},\h W_7)\left(\tau,\h{{y }}_2(\tau;y_1,y'),\h{{ y}}_3(\tau;y_1,y')\right)\de \tau\\  \end{aligned}
	\end{equation*}\begin{equation}\label{6-10}\begin{aligned}
			&=\frac{b_2}{b_1}W_1(r_b,y')+b_2\left( W_7(\eta_2 (y_1,y'), \eta_3 (y_1,y'))-W_7(y')\right)\\
			&\quad+R_3\left(\h{\bf W}(r_b,\eta_2 (y_1,y'), \eta_3 (y_1,y')),\h W_7(\eta_2 (y_1,y'), \eta_3 (y_1,y'))\right) \\
			&\quad+\int_{r_b}^{y_1} I_6(W_6,\h{\bf W},\h W_7)\left(\tau,\h{{y }}_2(\tau;y_1,y'),\h{{ y}}_3(\tau;y_1,y')\right)\de \tau,\\
		\end{aligned}
	\end{equation}
	where
	\begin{equation*}
	R_3 ({\bf \hat{W}} (r_b, y'), \hat{W}_7 (y')) = - \f {b_2}{b_1} R_1 ({\bf \hat{W}} (r_b, y'), \hat{W}_7 (y')) + R_2({\bf \hat{W}} (r_b, \eta_2, \eta_3), \hat{W}_7 (\eta_2, \eta_3) ).
	\end{equation*}
	Since $W_7 (y')$ is still unknown, \eqref{6-10} can be rewritten as
	\begin{equation}\label{6-11}
		\begin{aligned}
			W_4(y_1,y')&= \f {b_2}{b_1} W_1 (r_b, y') +R_4 \left(\h{\bf W}(r_b,\eta_2 (y_1,y'), \eta_3 (y_1,y')),\h W_7(\eta_2 (y_1,y'), \eta_3 (y_1,y'))\right)\\
			&\quad+\int_{r_b}^{y_1} I_6(W_6,\h{\bf W},\h W_7)\left(\tau,\h{{y }}_2(\tau;y_1,y'),\h{{ y}}_3(\tau;y_1,y')\right)\de \tau,\\
		\end{aligned}
	\end{equation}
	with \begin{equation*}
		\begin{aligned}
			R_4(\h{\bf W},\h W_7)  &= b_2 \left( \h W_7 (\eta_2 (y_1,y'), \eta_3 (y_1,y'))-\h W_7 (y') \right)\\
			&\quad+ R_3 \left(\h{\bf W}(r_b,\eta_2 (y_1,y'), \eta_3 (y_1,y')),\h W_7(\eta_2 (y_1,y'), \eta_3 (y_1,y'))\right).
	\end{aligned}\end{equation*}
	Therefore, it holds that
	\begin{equation}\label{6-12}
		\|W_5\|_{C^{2,\alpha}(\overline\mn)} \leq C_*\sigma \left( \| \hat{W}_7 \|_{C^{2,\alpha}(\overline{D})} + \sum_{i= 2}^{3} \| \eta_i \|_{C^{2, \alpha} (\overline{\mn})} \right)
		+C_* \|I_5\|_{C^{2,\alpha}(\overline\mn)}
		\leq  C_*(\sigma \delta_0+ \delta_0^2).
	\end{equation}
	and
	\begin{equation}\label{3-3}
		\begin{aligned}
			\|W_4\|_{C^{2,\alpha}(\overline\mn)}\leq & C_*\left(\|W_1(r_b,\cdot)\|_{C^{2,\alpha}(\overline{D})}
			+\|R_4\|_{C^{2,\alpha}(\overline\mn)}+\|I_6\|_{C^{2,\alpha}(\overline\mn)}\right)\\
			\leq & C_*\bigg(\|W_1(r_b,\cdot)\|_{C^{2,\alpha}(\overline{D})}
			+\|\hat{W}_7\|_{C^{3,\alpha}(\overline{D})}\sum_{j=2}^3
			\|\eta_j(y)-y_j\|_{C^{2,\alpha}(\overline\mn)}\bigg)\\
			&+ C_*\left(\sigma \|(\hat{{\bf W}}, \hat{W}_7)\|_{\mathfrak{W}}+\|(\hat{{\bf W}}, \hat{W}_7)\|_{\mathfrak{W}}^2\right)\\
			\leq & C_*\|W_1(r_b,\cdot)\|_{C^{2,\alpha}(\overline D)}+ C_*(\sigma \delta_0+ \delta_0^2).
		\end{aligned}
	\end{equation}
	Furthermore, it follows from \eqref{1-11-SUP}, \eqref{bc20},   \eqref{bc24} and \eqref{bc21-hy} that
	\be\label{W5-cp}
	\begin{cases}
		\p_{y_2} W_5 (y_1,\pm\th_0, y_3) = 0 , & (y_1, y_3) \in \Sigma_{\th_0}^\pm, \\
		\p_{y_3}  W_5 (y_1,y_2, \pm L) = 0 , & (y_1, y_2) \in \Sigma_{L}^\pm.
	\end{cases}
	\ee
	\par Since $(\hat{{\bf W}}, \hat{W}_7)\in\mathfrak{W}$ satisfies the compatibility condition \eqref{bc20} and the incoming supersonic flow satisfies \eqref{1-8-SUP},  one could use  the formulas of $\ma_2, \ma_3, \ma, R_{1i}$ $(i=1,2,3)$  to verify that
	\be\label{j211}\begin{cases}
		\ma_2(\hat{{\bf W}}(r_b,y'),\hat{W}_7(y'))\bigg|_{y_2=\pm \th_0}=\p_{y_2}^2\left(\ma_2(\hat{{\bf W}}(r_b,y'),\hat{W}_7 (y'))\right)\bigg|_{y_2=\pm \th_0}=0,   &y_3\in [-L,L],\\
		\p_{y_2}\left(\ma_3(\hat{{\bf W}}(r_b,y'),\hat{W}_7(y'))\right)\bigg|_{y_2=\pm \th_0}=  \p_{y_2}\left(\ma(\hat{{\bf W}}(r_b,y'),\hat{W}_7(y'))\right)\bigg|_{y_2=\pm \th_0}=0, &y_3\in [-L,L],\\
		\ma_3(\hat{{\bf W}}(r_b,y'),\hat{W}_7 (y')))\bigg|_{y_3=\pm L}=\p_{y_3}^2\left(\ma_3(\hat{{\bf W}}(r_b,y'),\hat{W}_7 (y'))\right)\bigg|_{y_3=\pm L}=0, &y_2\in [-\th_0,\th_0],\\
		\p_{y_3}\left(\ma_2(\hat{{\bf W}}(r_b,y'),\hat{W}_7 (y'))\right)\bigg|_{y_3=\pm L}=  \p_{y_3}\left(\ma(\hat{{\bf W}}(r_b,y'),\hat{W}_7(y')) \right)\bigg|_{y_3=\pm L}=0,  &y_2\in [-\th_0,\th_0],\\
		\p_{y_2}\left( R_{1i}(\hat{{\bf W}}(r_b,y'),\hat{W}_7 (y'))\right)\bigg|_{y_2=\pm \th_0}=0,  &y_3\in [-L,L],\\
		\p_{y_3}\left( R_{1i}(\hat{{\bf W}}(r_b,y'),\hat{W}_7(y'))\right)\bigg|_{y_3=\pm L}=0,\ \  &y_2\in [-\th_0,\th_0].
	\end{cases}\ee
	Thus for $k=1,2,3$,
	\be\label{rv41}\begin{cases}
		\p_{y_2} \left(R_{k}(\hat{{\bf W}}(r_b,y'),\hat{W}_7 (y'))\right)\bigg|_{y_2=\pm \th_0}=0,&  y_3\in [-L,L],\\
		\p_{y_3}\left(R_{k}(\hat{{\bf W}}(r_b,y'),\hat{W}_7(y'))\right)\bigg|_{y_3=\pm L}=0,& y_2\in [-\th_0,\th_0].
	\end{cases}\ee
	These, together with   \eqref{bc23}, \eqref{bc24} and \eqref{bc21-hy}, yield that
	\be\label{rv42}
	\begin{cases}
		\p_{y_2}\bigg(R_4(\h{\bf W}(r_b,\eta_2 (y_1,y'), \eta_3 (y_1,y')),\\
		\qquad\quad\h W_7(\eta_2 (y_1,y'), \eta_3 (y_1,y'))\bigg)(y_1,\pm \th_0, y_3)=0, \quad (y_1, y_3) \in\Sigma_{\th_0}^\pm, \\
		\p_{y_3}\bigg(R_4(\h{\bf W}(r_b,\eta_2 (y_1,y'), \eta_3 (y_1,y')),\\
		\qquad\quad\h W_7(\eta_2 (y_1,y'), \eta_3 (y_1,y'))\bigg)(y_1, y_2,\pm L)=0, \quad (y_1, y_2) \in\Sigma_{L}^\pm, \\
	\end{cases}\ee
	and
	\be\label{rv43}
	\begin{cases}
		\p_{y_2}W_4(y_1,\pm \th_0, y_3)=\f {b_2}{b_1} \p_{y_2}W_1 (r_b,\pm \th_0, y_3)=0,\ \ (y_1, y_3) \in \Sigma_{\th_0}^\pm, \\
		\p_{y_3}W_4(y_1, y_2,\pm L)=\f {b_2}{b_1} \p_{y_3}W_1(r_b, y_2,\pm L)=0,\ \ \ (y_1, y_2) \in \Sigma_{L}^\pm.
	\end{cases}\ee
	
	\par {\bf Step 4. Solving the transport equation for $\e\om_1$.}
	\par Due to \eqref{om7} and \eqref{bc6}, it suffices to consider the following problem
	\begin{equation}\label{om7-co}
		\begin{cases}
			\left(\p_{y_1} + I_2(\h{\bf W},\h W_7)\p_{y_2} + I_3(\h{\bf W},\h W_7)\p_{y_3} \right) \e{\om}_1
			+  \mu(\h{\bf W},\h W_7)\e\om_1
			= G_1 (\h{\bf W}, \h W_7),\\
			\tilde{\omega}_1(r_b,y')=R_5(\hat{{\bf W}}(r_b,y'), \hat{W}_7(y')),
		\end{cases}
	\end{equation}
	where
	\begin{align*}
			& \mu (\hat{{\bf W}}, \hat{W}_7) = \frac{  (r_1 - \hat{W}_7- r_b) \left(\bar{U} + \hat{W}_1\right)  J_0^{\hat{W}_7} }{\left( \bar U+ \hat{W}_1 \right) J_0^{\hat{W}_7} (r_1 - r_b)  + (y_1 - r_1) \bigg( \hat{W}_2 \p_{y_2} \hat{W}_7 + J_0^{\hat{W}_7} \hat{W}_3 \p_{y_3} \hat{W}_7 \bigg)}\\
			&\qquad\qquad\quad \times \bigg( \f1{J_0^{\hat{W}_7}}+ J^{\hat{W}_7}_2 \bigg( \f {\hat{W}_2}{\bar{U}  + \hat{W}_1}\bigg) + J^{\hat{W}_7}_3 \bigg( \f {\hat{W}_3}{\bar{U} + \hat{W}_1} \bigg) \bigg), \\
			& G_1 (\hat{{\bf W}}, \hat{W}_7) = \frac{ (r_1 - \hat{W}_7- r_b) J_0^{\hat{W}_7}  \left( \bar{U} + \hat{W}_1 \right) H_1 (\hat{{\bf W}}, \hat{W}_7)}{\left( \bar U+ \hat{W}_1 \right) J_0^{\hat{W}_7} (r_1 - r_b)  + (y_1 - r_1) \left( \hat{W}_2 \p_{y_2} \hat{W}_7 + J_0^{\hat{W}_7} \hat{W}_3 \p_{y_3} \hat{W}_7 \right)},\\
			& R_5(\hat{{\bf W}}(r_b,y'), \hat{W}_7(y')) = \frac 1{b_0} \left( \p_{y_3} g_2 (\h{\bf W} (r_b, y'),\h W_7 (y')) - \frac1{r_b}\p_{y_2} g_3 (\h{\bf W} (r_b, y'), \h W_7 (y')) \right)\\
			& \qquad\qquad\qquad\qquad\quad
			-\frac{ \h W_7\p_{y_2}\h W_3(r_b,y')}{r_b(r_b+\h W_7)}-\frac{(r_b-r_1)
				\p_{y_3}\h W_7\p_{y_1}\h W_2(r_b,y')}{r_1-r_b-\h W_7}\\
			& \qquad\qquad\qquad\qquad\quad
			+\frac{r_b-r_1}{r_b+\h W_7}\frac{\p_{y_2}\h W_7 \p_{y_1}\h W_3(r_b,y')}{r_1-r_b-\h W_7}.
	\end{align*}
	Since $(\hat{{\bf W}},\hat{W}_7)\in\mathfrak{W}$ satisfies the compatibility condition \eqref{bc20}, using the first formula in \eqref{rh11-1} and \eqref{bc6-bo}, one can verify that
	\be\label{vor5011}\begin{cases}
		\tilde{\omega}_1(r_b,\pm \theta_0, y_3)=0, &  y_3\in [-L,L],\\
		\tilde{\omega}_1(r_b,y_2,\pm L)=0,&  y_2\in [-\theta_0,\theta_0],\\
		G_1(\hat{{\bf W}},\hat{W}_7)(y_1,\pm\theta_0, y_3)=0, & (y_1, y_3) \in \Sigma_{\th_0}^\pm,\\
		G_1(\hat{{\bf W}},\hat{W}_7)(y_1,y_2,\pm L)=0, & (y_1, y_2) \in \Sigma_{L}^\pm.
	\end{cases}\ee
	Integrating the equation in \eqref{om7-co} along the trajectory $(\tau,\h{{ y}}_2(\tau;y_1,y'),\h{{y}}_3(\tau;y_1,y'))$ yields
	\begin{equation}\label{vor502}
		\begin{split}
			\tilde{\omega}_1(y_1,y')
			&= \exp\left( -\!\int_{r_b}^{y_1} \mu\big(t, \hat{y}_2(t; y_1, y'), \hat{y}_3(t; y_1, y')\big)\, dt \right) \\
			&\quad \times \Bigg[
			R_5\bigg(\h{\bf W}(r_b,\eta_2 (y_1,y'), \eta_3 (y_1,y')),\h W_7(\eta_2 (y_1,y'), \eta_3 (y_1,y'))\bigg) \\
			&\quad\quad + \int_{r_b}^{y_1}
			\exp\!\bigg( \int_{r_b}^s \mu\big(t, \hat{y}_2(t; y_1, y'), \hat{y}_3(t; y_1, y')\big)\, dt \bigg)
			G_1\big( s, \hat{y}_2(s; y_1, y'), \hat{y}_3(s; y_1, y') \big) \, ds
			\Bigg].
		\end{split}
	\end{equation}
	Thus the following estimate holds
	\begin{equation}\label{vor503}
			\|\tilde{\omega}_1\|_{C^{1,\alpha}(\overline\mn)}\leq C_*\left(\|\tilde{\omega}_1(r_b,\cdot)\|_{C^{1,\alpha}(\overline{D})}+\|G_1(\hat{{\bf W}},\hat{W}_7)\|_{C^{1,\alpha}(\overline\mn)}\right)
			\leq C_*(\sigma + \sigma\delta_0+\delta_0^2).
	\end{equation}
	Moreover, using \eqref{bc23}, \eqref{bc24}, \eqref{vor5011} and \eqref{vor502}, the following compatibility conditions hold:
	\be\label{vor5041}\begin{cases}
		\tilde{\omega}_1(y_1,\pm \th_0, y_3)=0,& (y_1, y_3) \in \Sigma_{\th_0}^\pm,\\
		\tilde{\omega}_1(y_1,y_2, \pm L)=0,& (y_1, y_2) \in \Sigma_{L}^\pm.
	\end{cases}\ee
	\par Substituting \eqref{6-11} and \eqref{vor502} into \eqref{euler9}-\eqref{euler11} yields
	\be\label{vor504}
	&&\frac{1}{y_1}\p_{y_2} W_3-\p_{y_3} W_2=G_2(\hat{{\bf W}},\hat{W}_7),\\\label{vor505}
	&&\p_{y_3} W_1- \p_{y_1} W_3 + d_5 (y_1) \p_{y_3} W_1(r_b,y')=G_3(W_5,W_6, \hat{{\bf W}},\hat{W}_7),\\\label{Wor506}
	&&\p_{y_1} W_2+\frac{W_2}{y_1}-\frac{1}{y_1} \p_{y_2} W_1 -\frac{1}{y_1} d_5 (y_1) \p_{y_2}W_1(r_b,y')=G_4(W_5, W_6,\hat{{\bf W}},\hat{W}_7),
	\ee
	where
	\begin{equation*}
		\begin{aligned}
			& d_5 (y_1) = \frac{b_2}{b_1}\f  {e^{S_0}\bar \rho^{\gamma-1}(y_1)}{(\gamma-1)\bar{U} (y_1)}, \ \
			G_2(\hat{{\bf W}},\hat{W}_7)=\tilde{\omega}_1(y)+H_2(\hat{{\bf W}},\hat{W}_7),\\
			&G_3(W_5,W_6,\hat{{\bf W}},\hat{W}_7)=\frac{\hat{W}_2\tilde{\omega}_1+J_3^{\hat{W}_7} W_5}{\bar{U}(J_0^{\hat{W}_7})+\hat{W}_1}+ H_3(\hat{{\bf W}},\hat{W}_7)- \f  {e^{S_0}\bar \rho^{\gamma-1}(y_1)}{(\gamma-1)\bar{U} (y_1)}\p_{y_3}   \left(R_4(\hat{{\bf W}},\hat{W}_7)\right)\\
			&\qquad\qquad\qquad\qquad\quad- \f  {e^{S_0}\bar \rho^{\gamma-1}(y_1)}{(\gamma-1)\bar{U} (y_1)}\int_{r_b}^{y_1} \p_{y_3}\left(I_6(W_6,\h{\bf W},\h W_7)\right)\de \tau,\\\end{aligned}
	\end{equation*}\begin{equation*}
		\begin{aligned}
			&G_4(W_5,W_6,\hat{{\bf W}},\hat{W}_7)=\frac{\hat{W}_3\tilde{\omega}_1-J_2^{\hat{W}_7} W_5}{\bar{U}(J_0^{\hat{W}_7})+\hat{W}_1}+H_4(\hat{{\bf W}},\hat{W}_7)+\f  {e^{S_0}\bar \rho^{\gamma-1}(y_1)}{(\gamma-1)y_1\bar{U} (y_1)}\p_{y_2}  \left(R_4(\hat{{\bf W}},\hat{W}_7)\right)\\
			&\qquad\qquad\qquad\qquad\quad + \f  {e^{S_0}\bar \rho^{\gamma-1}(y_1)}{(\gamma-1)y_1\bar{U} (y_1)}\int_{r_b}^{y_1} \p_{y_2}\left(I_6(W_6,\h{\bf W},\h W_7)\right)\de \tau.
		\end{aligned}
	\end{equation*}
	Using \eqref{bc20}, \eqref{W6-cp},  \eqref{bc21-hy}, \eqref{W5-cp}, \eqref{rv42} and \eqref{vor5041}, one can further verify the following compatibility conditions
	\be\label{cp20}\begin{cases}
		\left(G_2(\h{\bf W}, \h W_7), G_4(\h{\bf W}, \h W_7),\p_{y_2}G_3(\h{\bf W}, \h W_7)\right)(y_1,\pm \th_0, y_3)=0,&( y_1, y_3) \in \Sigma_{\th_0}^\pm,\\
		\left(G_2(\h{\bf W}, \h W_7), G_3(\h{\bf W}, \h W_7),\p_{y_3}G_4(\h{\bf W}, \h W_7)\right)(y_1,y_2, \pm L)=0,& (y_1, y_2) \in \Sigma_{L}^\pm.
	\end{cases}\ee
	Furthermore, \eqref{rho2-1-co} implies that
	\begin{equation}\label{rho2-1-ell}
		\begin{aligned}
			d_1(y_1)\p_{y_1} W_1 + \frac1{y_1}\p_{y_2} W_2 +   \p_{y_3} W_3+\frac{ W_1}{y_1}+d_2(y_1)W_1=
			G_0(W_5,W_6,\h{\bf W},\h W_7),
		\end{aligned}
	\end{equation}
	where
	$$G_0(W_5,W_6,\h{\bf W},\h W_7)=d_3(y_1)W_5+d_4(y_1)W_6
	+H_0(\h{\bf W},\h W_7).$$
	It follows from \eqref{2-17-ex} and \eqref{6-11} that the boundary condition for $W_1$ at the exit becomes
	\be\label{pres4}
	W_1(r_1,y')+ d_5 (r_1) W_1(r_b,y')=q_4(y'),
	\ee
	where
	\begin{equation*}
		\begin{aligned}
			q_4(y')&=-\f  {e^{S_0}\bar \rho^{\gamma-1}(r_1)}{(\gamma-1)\bar{U} (r_1)} R_4 \bigg(\h{\bf W}(r_b,\eta_2 (r_1,y'), \eta_3 (r_1,y')),\h W_7(\eta_2 (r_1,y'), \eta_3 (r_1,y'))\bigg)\\
			&\quad-\f  {e^{S_0}\bar \rho^{\gamma-1}(r_1)}{(\gamma-1)\bar{U} (r_1)}\int_{r_b}^{r_1} I_6(W_6,\h{\bf W},\h W_7)\bigg(\tau,\h{{y }}_2(\tau;y_1,y'),\h{{ y}}_3(\tau;y_1,y')\bigg)\de \tau
			\\
			&\quad +\frac{W_5(r_1,y')}{\bar{U}(r_1)}-\frac{\sigma P_{ex}(y')}{(\bar{\rho}\bar{U})(r_1)}-\frac{1}{2\bar{U}(r_1)}\sum_{j=1}^3 \hat{W}_j^2(r_1,y')-\frac{E(\hat{{\bf W}}(r_1,y'))}{\bar{U}(r_1)} .
		\end{aligned}
	\end{equation*}
	Using \eqref{1-13-sub} and the explicit expression of $E(\hat{{\bf W}}(r_1,y'))$ in \eqref{2-17}, one can verify that
	\be\label{err4}\begin{cases}
		\p_{y_2} q_4(\pm\theta_0, y_3)=0,&y_3\in [-L,L],\\
		\p_{y_3} q_4(y_2, \pm L)=0,& y_2\in [-\theta_0,\theta_0].
	\end{cases}\ee
	\par {\bf Step 5. Solving the deformation-curl system for the velocity field.}
	\par In  {\bf Step 4}, we derived the deformation-curl system for the velocity field, which consists of equations \eqref{vor504}-\eqref{Wor506}, \eqref{rho2-1-ell} and boundary conditions \eqref{shock19}-\eqref{bc17}, \eqref{slip16}, \eqref{pres4}, with $q_i^{\pm}$ $(i = 2,3)$ and $q_i$ $(i = 1, 4)$ taking values at $(\hat{\mathbf{W}}, \hat{W}_7)$. However, due to the linearization,  $(G_2, G_3, G_4)$ generally does not satisfy the divergence-free condition, the solvability condition for the system \eqref{vor504}-\eqref{Wor506} is generally invalid. Therefore, we first consider the following enlarged deformation-curl system by introducing an additional unknown function $\Pi$ that satisfies the homogeneous Dirichlet boundary condition:
	\begin{equation}\label{defor-curl-Pi}
		\begin{cases}
			d_1(y_1)\p_{y_1} W_1 + \frac1{y_1}\p_{y_2} W_2 +   \p_{y_3} W_3+\frac{ W_1}{y_1}+d_2(y_1)W_1
			=
			G_0(W_5,W_6,\h{\bf W},\h W_7),  &\text{in } \mn, \\
			\frac{1}{y_1}\p_{y_2} W_3-\p_{y_3} W_2 + \partial_{y_1} \Pi
			=G_2(\hat{{\bf W}},\hat{W}_7),  &\text{in } \mn, \\
			\p_{y_3} W_1- \p_{y_1} W_3 + d_5 (y_1) \p_{y_3} W_1(r_b,y') + \frac{1}{y_1} \partial_{y_2} \Pi
			=G_3(W_5,\hat{{\bf W}},\hat{W}_7),  &\text{in } \mn,\\
			\p_{y_1} W_2+\frac{W_2}{y_1}-\frac{1}{y_1} \p_{y_2} W_1 -\frac{1}{y_1} d_5 (y_1) \p_{y_2}W_1(r_b,y') + \partial_{y_3} \Pi
			=G_4(W_5,\hat{{\bf W}},\hat{W}_7),  &\text{in } \mn,\\
			\frac{1}{r_b^2}\p_{y_2}^2 W_1 (r_b,y')+\p_{y_3}^2 W_1 (r_b,y') \\
			= b_0b_1\left(\frac{1}{r_b}\p_{y_2} W_2+\p_{y_3} W_3\right)(r_b,y')+ q_1(\hat{{\bf W}}(r_b,y'), \hat{W}_7 (y')),  &y' \in D, \\
			W_2 (y_1, \pm \theta_0, y_3) = \Pi (y_1, \pm \theta_0, y_3) = 0,  &(y_1, y_3) \in \Sigma_{\th_0}^\pm, \\
			W_3 (y_1, y_2, \pm L ) = \Pi (y_1, y_2, \pm L ) = 0,  &(y_1, y_2) \in \Sigma_{L}^\pm, \\
			W_1 (r_1, y') + d_5 (r_1) W_1 (r_b, y') = q_4 (y'), &y' \in D, \\
			\left(\frac1{r_b}\p_{y_2} W_1 - b_0 b_1 W_2\right) (r_b, \pm \th_0, y_3 ) = 0, &y_3 \in [-L,L], \\
			(\p_{y_3} W_1 - b_0 b_1 W_3) (r_b,  y_2, \pm L) = 0,   &y_2 \in [-\th_0,\th_0], \\
			\partial_{y_1}\Pi (r_b, y') = \Pi (r_1, y') = 0,  &y' \in D.
		\end{cases}
	\end{equation}
	
	{\bf Step 5.1.} First, taking the divergence operator to the second, third, and fourth equations of system \eqref{defor-curl-Pi}, and combining with the boundary conditions of $\Pi$ lead to
	\begin{equation}\label{Pi-1}
		\begin{cases}
			\p_{y_1 }^2 \Pi + \frac{1}{y_1} \p_{y_1} \Pi + \frac{1}{y_1^2} \p_{y_2}^2 \Pi + \p_{y_3}^2 \Pi \\
			\quad = \p_{y_1} G_2 + \frac{1}{y_1} G_2 + \frac{1}{y_1} \p_{y_2} G_3 + \p_{y_3} G_4,\quad & \mbox{in } \mn, \\
			\Pi (y_1, \pm \theta_0, y_3) = 0, &(y_1, y_3) \in \Sigma_{\th_0}^\pm, \\
			\Pi (y_1, y_2, \pm L ) = 0, &(y_1, y_2) \in \Sigma_{L}^\pm, \\
			\partial_{y_1} \Pi (r_b, y') = \Pi (r_1, y') = 0, & y' \in D.
		\end{cases}
	\end{equation}
	By applying the standard elliptic theory in \cite{GT98}, there exists a unique solution $\Pi \in C^{2, \alpha}(\mn) \cap C^{1, \alpha} (\overline{\mn})$ to the elliptic system \eqref{Pi-1}. To deal with the singularity near the corner, we extend $\Pi$ and $G_i$ $(i = 2, 3, 4)$  as follows:
	\begin{equation}\label{extension-1}
		\left(\tilde{\Pi}, \tilde{G}_2, \tilde{G}_3,\tilde G_4 \right) (y_1, y')  =\begin{cases}
			\left(-\Pi, -G_2, -G_3,G_4 \right) (y_1, y_2, - y_3 - 2L), &  y'  \in  [- \theta_0, \theta_0] \times [-3L, -L], \\
			\left(\Pi, G_2, G_3,G_4 \right) (y_1, y_2, y_3), &  y'  \in  [- \theta_0, \theta_0] \times [-L, L], \\
			\left(-\Pi, -G_2, -G_3,G_4 \right) (y_1, y_2, - y_3 + 2L), &  y' \in   [- \theta_0, \theta_0] \times [L, 3L].
		\end{cases}
	\end{equation}
		The similar extension procedure is performed for $\Pi$ and $G_i \ (i = 2, 3, 4)$ along the $x_2$-direction. Combining the compatibility condition \eqref{cp20}, one has
		\begin{equation}\label{Pi-2}
			\begin{cases}
				\p_{y_1}^2 \tilde{\Pi} + \frac{1}{y_1} \p_{y_1} \tilde{\Pi} + \frac{1}{y_1^2} \p_{y_2}^2 \tilde{\Pi} + \p_{y_3}^2 \tilde{\Pi}
				= \p_{y_1} \tilde{G}_2 + \frac{1}{y_1} \tilde{G}_2 + \frac{1}{y_1} \p_{y_2} \tilde{G}_3 + \p_{y_3} \tilde{G}_4 \in C^{0, \alpha} (\overline{\mn_e}),  \\
				\tilde{\Pi} (r_b, y_2, y_3) = \tilde{\Pi} (r_1, y_2, y_3) = 0, \quad (y_2, y_3) \in [- 3 \theta_0, 3 \theta_0] \times [-3L, 3L],
			\end{cases}
		\end{equation}
		where $\mn_e = (r_b, r_1) \times (- 3 \theta_0, 3 \theta_0) \times (-3L, 3L)$. Then the regularity of $\Pi$ on $\overline\mn$ can be improved  to $C^{2, \alpha} (\overline{\mn})$ with the estimate
		\begin{equation}\label{Pi-est}
			\| \Pi \|_{C^{2, \alpha} (\overline{\mn})} \leq C_* \sum_{i = 2}^4 \| G_i \|_{C^{1, \alpha} (\overline{\mn})} \leq C_* \left(\sigma \| (\hat{\mathbf{W}}, \hat{W}_7) \|_{\mathfrak{W}}  + \| (\hat{\mathbf{W}}, \hat{W}_7) \|_{\mathfrak{W}}^2 \right) \leq C_* (\sigma \delta_0 + \delta_0^2).
		\end{equation}
		Furthermore, the following compatibility conditions hold:
		\begin{equation}\label{Pi-cp}
			\begin{cases}
				\p_{y_1} \Pi (y_1, \pm \theta_0, y_3) = \p_{y_3} \Pi (y_1, \pm \theta_0, y_3) = 0, &  (y_1,  y_3)\in\Sigma_{\th_0}^\pm, \\
				\p_{y_1} \Pi (y_1, y_2, \pm L) = \p_{y_2} \Pi (y_1, y_2, \pm L) = 0, &  (y_1,  y_2)\in \Sigma_{L}^\pm .
			\end{cases}
		\end{equation}
		
		{\bf Step 5.2.} We now turn to solving the following divergence-curl system with homogeneous normal boundary conditions:
		\begin{equation}\label{div-curl}
			\begin{cases}
				\p_{y_1} \tilde{W}_1 + \frac{1}{y_1} \tilde{W}_1 + \frac{1}{y_1}  \p_{y_2} \tilde{W}_2 + \p_{y_3} \tilde{W}_3 = 0, & \mbox{in } \mn, \\
				\frac{1}{y_1} \p_{y_2} \tilde{W}_3 - \p_{y_3} \tilde{W}_2 = G_2 (\hat{\mathbf{W}}, \hat{W}_7) - \p_{y_1} \Pi \coloneqq \bar{G}_2, & \mbox{in } \mn, \\
				\p_{y_3} \tilde{W}_1 - \p_{y_1} \tilde{W}_3 = G_3 (\hat{\mathbf{W}}, \hat{W}_7) - \frac{1}{y_1} \p_{y_2} \Pi \coloneqq \bar{G}_3 , & \mbox{in } \mn, \\
				\p_{y_1} \tilde{W}_2 + \frac{1}{y_1} \tilde{W}_2 - \frac{1}{y_1} \p_{y_2} \tilde{W}_1 = G_4 (\hat{\mathbf{W}}, \hat{W}_7) - \p_{y_3} \Pi \coloneqq \bar{G}_4, & \mbox{in } \mn, \\
				\tilde{W}_1 (r_b, y_2, y_3) = \tilde{W}_1 (r_1, y_2, y_3) = 0, &  y'\in D, \\
				\tilde{W}_2 (y_1, \pm \theta_0, y_3) = 0, & (y_1,  y_3) \in\Sigma_{\theta_0}^{\pm}, \\
				\tilde{W}_3 (y_1, y_2, \pm L) = 0, &  (y_1,  y_2) \in \Sigma_L^{\pm}.
			\end{cases}
		\end{equation}
		It follows from \eqref{Pi-1} that
		\begin{equation*}
			\p_{y_1} \bar{G}_2 + \frac{1}{y_1} \bar{G}_2 + \frac{1}{y_1} \p_{y_2} \bar{G}_3 + \p_{y_3} \bar{G}_4 \equiv 0, \quad \text{in } \mn.
		\end{equation*}
		Moreover, \eqref{cp20} and \eqref{Pi-2} imply that
		\begin{equation}\label{G-cp}
			\begin{cases}
				\left(\bar{G}_2(\h{\bf W}, \h W_7), \bar{G}_4(\h{\bf W}, \h W_7),\p_{y_2} \bar{G}_3(\h{\bf W}, \h W_7)\right)(y_1,\pm \th_0, y_3)=0,&(y_1,  y_3) \in \Sigma_{\theta_0}^{\pm},\\
				\left(\bar{G}_2(\h{\bf W}, \h W_7), \bar{G}_3(\h{\bf W}, \h W_7),\p_{y_3} \bar{G}_4(\h{\bf W}, \h W_7)\right)(y_1,y_2, \pm L)=0,& (y_1,  y_2) \in \Sigma_L^{\pm}.
			\end{cases}
		\end{equation}
		By applying the standard theory in \cite{KY09}, the divergence-curl system with the homogeneous boundary conditions is uniquely solvable. Thanks to the compatibility condition \eqref{G-cp}, one can use the  symmetric extension technique to show that there exists a unique $C^{2, \alpha} (\overline{\mn})$ smooth vector field $\left(\tilde{W}_1, \tilde{W}_2, \tilde{W}_3 \right)$  to the problem \eqref{div-curl} satisfying the estimate
		\begin{align} \nonumber
			& \sum_{i = 1}^3 \| \tilde{W}_i \|_{C^{2, \alpha} (\overline{\mn})} \leq C_* \sum_{i = 2}^4 \| \bar{G}_i \|_{C^{1, \alpha} (\overline{\mn})} \leq C_* \left( \sum_{i = 2}^4 \| G_i \|_{C^{1, \alpha} (\overline{\mn})} + \| \Pi \|_{C^{2, \alpha} (\overline{\mn})} \right) \\ \label{W-est-1}
			& \leq C_*  \sum_{i = 2}^4 \| G_i \|_{C^{1, \alpha} (\overline{\mn})} \leq  C_* \left(\sigma \| (\hat{\mathbf{W}}, \hat{W}_7) \|_{\mathfrak{W}}  + \| (\hat{\mathbf{W}}, \hat{W}_7) \|_{\mathfrak{W}}^2 \right) \leq C_* (\sigma \delta_0 + \delta_0^2),
		\end{align}
		and the compatibility conditions
		\begin{equation}\label{W-cp}
			\begin{cases}
				\left( \p_{y_2} \tilde{W}_1, \p_{y_2} \tilde{W}_3, \tilde{W}_2, \p_{y_2 y_2}^2 \tilde{W}_2 \right) (y_1, \pm \theta_0, y_3) = 0, & (y_1,  y_3) \in \Sigma_{\theta_0}^{\pm}, \\
				\left( \p_{y_3} \tilde{W}_1, \p_{y_3} \tilde{W}_2, \tilde{W}_3, \p_{y_3 y_3}^2 \tilde{W}_3 \right) (y_1, y_2,  \pm L) = 0, & (y_1,  y_2) \in\Sigma_L^{\pm}.
			\end{cases}
		\end{equation}
		
		{\bf Step 5.3.} Let $(W_1, W_2, W_3)$ be the solution to \eqref{defor-curl-Pi}. Set
		\begin{equation*}
			M_i (y_1,y') = W_i (y_1,y') - \tilde{W}_i (y_1,y'), \quad i =1,2,3.
		\end{equation*}
		Then $M_i \ (i = 1,2,3)$ satisfy
		\begin{equation}\label{M}
			\begin{cases}
				d_1 (y_1) \p_{y_1} M_1 + \frac{1}{y_1} \p_{y_2} M_2 + \p_{y_3} M_3 + \frac{1}{y_1} M_1 + d_2 (y_1) M_1 \\
				= G_5(W_5, W_6, \hat{\mathbf{W}}, \hat{W}_7), & \mbox{in } \mn, \\
				\frac{1}{y_1} \p_{y_2} M_3 - \p_{y_3} M_2 = 0, & \mbox{in } \mn, \\
				\p_{y_3} M_1 - \p_{y_1} M_3 + d_5 (y_1) \p_{y_3} M_1 (r_b, y') = 0, & \mbox{in } \mn, \\
				\p_{y_1} M_2 - \frac{1}{y_1} \p_{y_2} M_1 + \frac{1}{y_1} M_2 - \frac{1}{y_1} d_5 (y_1) \p_{y_2} M_1 (r_b, y') = 0, &  \mbox{in } \mn, \\
				\frac{1}{r_b^2} \p_{y_2 }^2 M_1 (r_b, y') + \p_{y_3 }^2 M_1 (r_b, y') \\
				- b_0 b_1 \left(\frac{1}{r_b} \p_{y_2} M_2 + \p_{y_3} M_3 \right) (r_b, y') = q_5 (y'), & y' \in D, \\
				M_2 (y_1, \pm \theta_0, y_3) = 0, &(y_1,  y_3) \in \Sigma_{\theta_0}^{\pm}, \\
				M_3 (y_1, y_2, \pm L) = 0, & (y_1,  y_2) \in \Sigma_L^{\pm}, \\
				M_1 (r_1, y') + d_5 (r_1) M_1 (r_b, y') = q_4 (y'), & y' \in D, \\
				\left(\frac{1}{r_b} \p_{y_2} M_1 - b_0 b_1 M_2 \right) (r_b, \pm \theta_0, y_3) = 0, & y_3 \in [-L, L], \\
				\left( \p_{y_3} M_1 - b_0 b_1 M_3 \right) (r_b, y_2, \pm L) = 0, & y_2 \in [- \theta_0, \theta_0],
			\end{cases}
		\end{equation}
		where
		\begin{align*}
			& G_5(W_5, W_6, \hat{\mathbf{W}}, \hat{W}_7) = G_0(W_5, W_6, \hat{\mathbf{W}}, \hat{W}_7) + \bar{M}^2 (y_1) \p_{y_1} \tilde{W}_1 - d_2 (y_1) \tilde{W}_1, \\
			& q_5 (y') = q_1 (\hat{\mathbf{W}}(r_b, y'), \hat{W}_7 (y')) + b_0 b_1 \left( \frac{1}{r_b} \p_{y_2} \tilde{W}_2 + \p_{y_3} \tilde{W}_3 \right) (r_b, y').
		\end{align*}
		It follows from the second, third and fourth equations in \eqref{M} that there exists a potential function $\varphi$ such that
		\begin{equation}\label{phi}\begin{cases}
				M_1  + d_5 (y_1) M_1 (r_b, y') = \p_{y_1} \varphi (y_1, y'), \\ M_2 (y_1, y') = \frac{1}{y_1} \p_{y_2} \varphi (y_1, y'), \\ M_3 (y_1, y') = \p_{y_3} \varphi (y_1, y').
			\end{cases}
		\end{equation}
		Thus one gets
		\begin{align*}
			M_1 (r_b, y') =  \frac{1}{b_3} \p_{y_1} \varphi (r_b, y'), \ \
			M_1 (y_1, y') =   \p_{y_1} \varphi (y_1, y') - \frac{d_5(y_1)}{b_3} \p_{y_1} \varphi (r_b, y'),
		\end{align*}
		where $$b_3 = \frac{1+ (\gamma - 1) \bar{M}^2 (r_b)}{\gamma \bar{M}^2 (r_b)} > 0.$$
		Therefore, \eqref{M} can be reformulated as
		\begin{equation}\label{phi-1}
			\begin{cases}
				d_1 (y_1) \p_{y_1}^2 \varphi + \frac{1}{y_1^2} \p_{y_2}^2 \varphi + \p_{y_3}^2 \varphi + \frac{1}{y_1} \p_{y_1} \varphi + d_2 (y_1) \p_{y_1} \varphi \\
				\quad - \frac{1}{b_3} d_6 (y_1) \p_{y_1} \varphi (r_b, y') = G_5 (W_5, W_6, \hat{\mathbf{W}}, \hat{W}_7), & \mbox{in } \mn, \\
				(\frac{1}{r_b^2}  \p_{ y_2}^2  + \p_{y_3}^2  )(\p_{y_1} \varphi (r_b, y') - b_4 (r_b) \varphi (r_b, y')) = b_3q_5 (y'), & y'\in D , \\
				\p_{y_1} \varphi (r_1, y') = q_4 (y'), & y'\in D, \\
				\p_{y_2} \varphi (y_1, \pm \theta_0, y_3) = 0, &  (y_1,  y_3)\in\Sigma_{\theta_0}^{\pm}, \\
				\p_{y_3} \varphi (y_1, y_2, \pm L) = 0, & (y_1,  y_2)\in \Sigma_L^{\pm}, \\
				\p_{y_2} \left( \p_{y_1} \varphi - b_4  \varphi \right) (r_b, \pm \theta_0, y_3) = 0, & y_3 \in [-L, L], \\
				\p_{y_3} \left( \p_{y_1} \varphi - b_4 \varphi \right) (r_b, y_2, \pm L) = 0, & y_2 \in [- \theta_0, \theta_0],
			\end{cases}
		\end{equation}
		where
		\begin{align*}
			d_6 (y_1) & =  \left( \frac{1}{y_1}  + d_2 (y_1) \right) d_5 (y_1) + d_1 (y_1) d_5' (y_1) \\
			& = \left( \frac{1}{y_1}  + d_2 (y_1) \right) d_5 (y_1) - d_1 (y_1) \frac{b_2}{ b_1} \frac{\bar{B} + \frac{1}{2} \bar{U}^2 (y_1)}{\gamma} \frac{\bar{U}' (y_1)}{\bar{U}^2 (y_1)} \\
			& = \frac{b_2}{y_1 \gamma b_1 \bar{U} (y_1)} \left( 2 \bar{B} + \frac{\bar{U}^2 (y_1) \left(2+(\gamma-1) \bar{M}^2 (y_1) \right)}{ (\gamma-1) (1- \bar{M}^2 (y_1)) } \right) > 0, \\
			b_4 & = b_0 b_1 b_3 > 0.
		\end{align*}
		By the second equation with the last two boundary conditions  in \eqref{phi-1},
		an oblique boundary condition for the potential $\varphi$ on the boundary $\{(r_b, y') \colon y' \in D\}$ is derived by solving the Poisson equation with the Neumann boundary conditions.
		\begin{lemma}
			One the shock front $\{ (r_b, y') : y' \in D \}$, there exists a unique $C^{2, \al } (\overline{D})$ function $h_1 (y')$ such that
			\begin{equation}\no
				\p_{y_1} \varphi (r_b, y') - b_4 \varphi (r_b, y') = h_1 (y'),
			\end{equation}
			where $h_1 (y')$ satisfies the Poisson equation with the homogeneous Neumann boundary conditions
			\begin{equation*}
				\begin{cases}
					\bigg(\frac1{r_b^2}\p_{y_2}^2 + \p_{y_3}^2\bigg) h_1 (y') = b_3 q_5 (y') , & \mbox{in } D, \\
					\p_{y_2} h_1 (\pm \theta_0, y_3) = 0, &  y_3 \in [-L, L], \\
					\p_{y_3} h_1 (y_2, \pm L) = 0, &  y_2 \in [-\theta_0, \theta_0],
				\end{cases}
			\end{equation*}
			and the condition
			\begin{equation}\no
				\iint_D h_1 (y') d y' = 0.
			\end{equation}
		\end{lemma}
		Then the problem \eqref{phi-1} can be reduced to the following second-order elliptic equation with a nonlocal term:
		\begin{equation}\label{pf13}
			\begin{cases}
				\partial_{y_1} \left(d_1 (y_1) \partial_{y_1} \varphi\right) + \frac{1}{y_1^2} \partial_{y_2}^2 \varphi + \partial_{y_3}^2 \varphi + d_7 (y_1) \partial_{y_1} \varphi \\
				\quad - b_0 b_1 d_6 (y_1) \varphi (r_b, y') = G_6 (W_5, W_6, \hat{\mathbf{W}}, \hat{W}_7), & {\rm{in}} \ \ \mn, \\
				\partial_{y_1} \varphi (r_b, y') - b_4 \varphi (r_b, y') = h_1 (y'), &y'\in D, \\
				\partial_{y_2} \varphi (y_1, \pm \theta_0, y_3) = 0, &(y_1, y_3)\in \Sigma_{\theta_0}^{\pm}, \\
				\partial_{y_3} \varphi (y_1, y_2, \pm L) = 0, & (y_1, y_2)\in \Sigma_L^{\pm}, \\
				\partial_{y_1} \varphi (r_1, y') = q_4 (y'), & y'\in D,
			\end{cases}
		\end{equation}
		where
		\begin{equation*}
			d_7 (y_1) = \frac{1}{y_1} + d_2 (y_1) - d_1' (y_1), \quad
			G_6 (W_5, W_6, \hat{\mathbf{W}}, \hat{W}_7) = G_5 (W_5, W_6, \hat{\mathbf{W}}, \hat{W}_7) + \frac{1}{b_3} d_6 (y_1) h_1 (y').
		\end{equation*}
		Simple calculations yield that
		\begin{equation}\label{bc43}
			\begin{cases}
				\partial_{y_2} G_6 (y_1, \pm \theta_0, y_3) = 0, & (y_1, y_3)\in \Sigma_{\theta_0}^{\pm}, \\
				\partial_{y_3} G_6 (y_1, y_2, \pm L) = 0, & (y_1, y_2)\in \Sigma_L^{\pm}, \\
				\partial_{y_2} h_1 (\pm \theta_0, y_3)= \partial_{y_2} q_4 (\pm \theta_0, y_3) = 0, &  y_3 \in [-L , L],  \\
				\partial_{y_3} h_1 (y_2, \pm L) = \partial_{y_3} q_4 (y_2, \pm L) = 0, &  y_2 \in [-\theta_0 , \theta_0] .
			\end{cases}
		\end{equation}
		Note that $d_1 (y_1) > 0 $ for any $y_1 \in [r_b, r_1]$. Then the  problem \eqref{pf13} can be solved by adapting the arguments in \cite[Proposition 3.4]{WX23}.
		\begin{proposition}
			Assume $G_6 \in C^{1, \alpha} (\overline{\mn})$ and $(h_1, q_4) \in \left(C^{2, \alpha} (\overline{D})\right)^2$ satisfies \eqref{bc43}. Then problem \eqref{pf13} admits a unique solution $\varphi \in C^{3, \alpha} (\overline{\mn})$ with the estimate
			\begin{equation}\label{est9}
				\| \varphi \|_{C^{3, \alpha} (\overline{\mn})} \le C_* \left( \| G_6 \|_{C^{1, \alpha} (\overline{\mn})} + \| h_1 \|_{C^{2, \alpha} (\overline{D})} + \| q_4 \|_{C^{2, \alpha} (\overline{D})} \right),
			\end{equation}
			where $C_*$ is a constant depending only on the background solution.
		\end{proposition}
		
		Thus $$M_1 (y_1,y') = \p_{y_1} \varphi(y_1,y') - \f {d_5 (y_1)}{b_3} \p_{y_1} \varphi (r_b, y'),\ \ M_2 (y_1,y') = \frac{1}{y_1} \p_{y_2} \varphi (y_1,y'), \ \ M_3 (y_1,y') = \p_{y_3} \varphi (y_1,y') $$ satisfy  \eqref{M}. Differentiating the first equation in \eqref{M} with respect to $y_2$ (resp. $y_3$) and evaluating at $y_2 = \pm \theta_0$ (resp. $y_3 = \pm L$), one gets
		\begin{equation}\no
			\begin{cases}
				\p_{y_2}^2 M_2 (y_1, \pm \theta_0, y_3) = 0, & (y_1, y_3)\in \Sigma_{\theta_0}^{\pm}, \\
				\p_{y_3}^2 M_3 (y_1, y_2, \pm L) = 0, & (y_1, y_2)\in  \Sigma_L^{\pm}.
			\end{cases}
		\end{equation}
		Hence \begin{equation}\begin{cases}\label{pf32}
				W_1 (y_1,y') = \tilde{W}_1 (y_1,y') + \p_{y_1} \varphi(y_1,y') - \f {d_5 (y_1)}{b_3} \p_{y_1} \varphi (r_b, y'), \\
				W_2 (y_1,y') = \tilde{W}_2 (y_1,y') + \frac{1}{y_1} \p_{y_2} \varphi (y_1,y'), \\
				W_3 (y_1,y') = \tilde{W}_3 (y_1,y') + \p_{y_3} \varphi (y_1,y'), \end{cases}
		\end{equation}
		solve the problem \eqref{defor-curl-Pi}, which satisfy the following estimate
		\begin{equation}\label{W123-est}
			\begin{split}
				\sum_{i = 1}^3 \| W_i \|_{C^{2, \alpha} (\overline{\mn})} & \leq C_* \left(  \sum_{i = 1}^3 \| \tilde{W}_i \|_{C^{2, \alpha} (\overline{\mn})} + \| \nabla \varphi \|_{C^{2, \alpha} (\overline{\mn})}  + \| \partial_{y_1} \varphi (r_b, \cdot) \|_{C^{2, \alpha} (\overline{D})}  \right) \\
				& \leq C_* \left( \sigma + \sigma \| ({\bf \hat{W}}, \hat{W}_7) \|_{\mathfrak{W}} + \| ({\bf \hat{W}}, \hat{W}_7) \|_{\mathfrak{W}}^2 \right) \leq C_* \left(\sigma + \delta_0^2 \right),
			\end{split}
		\end{equation}
		and the  compatibility  conditions
		\begin{equation}\label{W123-cp}
			\begin{cases}
				\left(W_2, \p_{y_2 }^2 W_2, \p_{y_2} W_1, \p_{y_2} W_3 \right) (y_1, \pm \theta_0, y_3) = 0,  & (y_1, y_3)\in \Sigma_{\theta_0}^{\pm}, \\
				\left(W_3, \p_{y_3 }^2 W_3, \p_{y_3} W_1, \p_{y_3} W_2 \right) (y_1, y_2, \pm L) = 0, & (y_1, y_2)\in \Sigma_L^{\pm}.
			\end{cases}
		\end{equation}

		\par {\bf Step 6. The  unique determination of   the entropy and the shock front.} \par According to \eqref{6-11},
		we can infer that the function $W_4$ can be uniquely determined by \eqref{pf32}. Furthermore, $W_4$ satisfies the following estimate
		\begin{equation}\label{W4-est}
			\| W_4 \|_{C^{2, \al} (\overline{\mn})} \le C_* \| W_1 (r_b, \cdot ) \|_{C^{2, \al } (\overline{D})} + C_* \bigg(\sigma \| ({\bf \hat{W}}, \hat{W}_7) \|_{\mathfrak{W}} + \| ({\bf \hat{W}}, \hat{W}_7) \|_{\mathfrak{W}}^2 \bigg) \le C_* (\sigma + \delta_0^2),
		\end{equation}
		and the compatibility conditions
		\be\label{bc46}
		\begin{cases}
			\p_{y_2} W_4 (y_1, \pm \theta_0, y_3) =  \p_{y_2} W_1 (r_b, \pm \theta_0, y_3) = 0, & (y_1, y_3)\in \Sigma_{\theta_0}^{\pm}, \\
			\p_{y_3} W_4 (y_1, y_2, \pm L) =  \p_{y_3} W_1 (r_b, y_2, \pm L) = 0, & (y_1, y_2)\in\Sigma_L^{\pm}.
		\end{cases}
		\ee
		In addition, it follows  from \eqref{6-2} that
		\begin{equation}\label{pf34}
			W_7 (y') = \f 1{b_1} W_1 (r_b, y') - \f 1{b_1} R_1 (\hat{{\bf W}} (r_b, y'), \hat{W}_7 (y')).
		\end{equation}
		This implies that $W_7 \in C^{2, \alpha} (\overline{D})$ satisfies
		\begin{equation}\label{bc47}
			\begin{cases}
				\p_{y_2} W_7 (\pm \theta_0, y_3) = 0, &  y_3 \in [-L,L], \\
				\p_{y_3} W_7 (y_2, \pm L) = 0, &  y_2 \in [-\theta_0,\theta_0].
			\end{cases}
		\end{equation}
		
		It remains to improve the regularity of  $ W_7 $ to be ${C^{3,\alpha}(\overline{D})}$. To this end, we introduce
		\begin{equation*}
			\begin{cases}
				F_2 (y') = \frac1{r_b}\p_{y_2} W_1 (r_b, y') - b_0 b_1 W_2 (r_b, y') - \frac1{r_b}\p_{y_2} R_1 (\hat{{\bf W}} (r_b, y'), \hat{W}_7 (y')) - b_1 g_2 (\hat{{\bf W}} (r_b, y'), \hat{W}_7 (y')), \\
				F_3 (y') = \p_{y_3} W_1 (r_b, y') - b_0 b_1 W_3 (r_b, y') - \p_{y_3} R_1 (\hat{{\bf W}} (r_b, y'), \hat{W}_7 (y')) - b_1 g_3 (\hat{{\bf W}} (r_b, y'), \hat{W}_7 (y')).
			\end{cases}
		\end{equation*}
		The boundary conditions in \eqref{M} imply that
		\begin{equation}\no
			\begin{cases}
				\p_{y_2} F_3 - \p_{y_3} F_2 = 0, & \text{in } D,\\
				\p_{y_2} F_2 + \p_{y_3} F_3 = 0, & \text{in } D, \\
				F_2 (\pm \theta_0 , y_3) = 0, &  y_3 \in [-L,L], \\
				F_3 (y_2, \pm L) = 0, &  y_2 \in [-\theta_0,\theta_0].
			\end{cases}
		\end{equation}
		Using Lemma \ref{lm1}, $F_2 = F_3 \equiv 0$ in $D$. This, together with \eqref{pf34}, yields that
		\begin{equation}\label{pf36}
			\begin{cases}
				\frac{1}{r_b}\p_{y_2} W_7 (y') = b_0 W_2 (r_b, y') + g_2 (\hat{{\bf W}} (r_b, y'), \hat{W}_7 (y')), & \mbox{in } D, \\
				\p_{y_3} W_7 (y') = b_0 W_3 (r_b, y') + g_3 (\hat{{\bf W}} (r_b, y'), \hat{W}_7 (y')), & \mbox{in } D.
			\end{cases}
		\end{equation}
		Therefore, $W_7 \in C^{3, \al} (\overline{D})$ satisfies  the following estimate
		\begin{equation}\label{est14}
			\begin{aligned}
				\| W_7 \|_{C^{3, \al} (\overline{D})}
				&\le C_* (\| W_1 (r_b, \cdot) \|_{C^{2, \al } (\overline{D})} + \| R_1 (\hat{{\bf W}} (r_b, \cdot), \hat{W}_7 (\cdot)) \|_{C^{2, \al } (\overline{D})}) \\
				& \q
				+ C_* \sum_{i = 2}^3 ( \| W_i (r_b, \cdot) \|_{C^{2, \al } (\overline{D})} + \| g_i (\hat{{\bf W}} (r_b, \cdot), \hat{W}_7(\cdot)) \|_{C^{2, \al } (\overline{D})} ) \\
				& \le
				C_* (\sigma + \sigma \| ({\bf \hat{W}}, \hat{W}_7) \|_{\mathfrak{W}} + \| ({\bf \hat{W}}, \hat{W}_7) \|_{\mathfrak{W}}^2)
				\le C_* (\sigma + \sigma \delta_0 + \delta_0^2).
			\end{aligned}
		\end{equation}
		Differentiating the first (second) equation in \eqref{pf36} with respect to $y_2$ (resp. $y_3$) twice and evaluating at $y_2 = \pm \theta_0$ (resp. $y_3 = \pm L$) together with \eqref{j211} and \eqref{W123-cp}, one gets
		\begin{equation}\label{bc48}
			\begin{cases}
				\p_{y_2}^3 W_7 ( \pm \theta_0, y_3) = 0, &  y_3 \in [-L ,L], \\
				\p_{y_3}^3 W_7 ( y_2, \pm L) = 0, &  y_2 \in [-\theta_0 ,\theta_0].
			\end{cases}
		\end{equation}
		Combining the estimates \eqref{W6-est}, \eqref{6-12}, \eqref{W123-est}, \eqref{W4-est} and \eqref{est14}, there holds
		\begin{equation}\no
			\| ({\bf W}, W_7) \|_{\mathfrak{W}} = \sum_{i = 1}^6 \| W_i \|_{C^{2, \al} (\overline{\mn})} + \| W_7 \|_{C^{3, \al} (\overline{D})} \le C_0^* (\sigma + \sigma \delta_0 + \delta_0^2) \le C_0^* (\sigma + \delta_0^2).
		\end{equation}
		Choosing $\delta_0 = \sqrt{\sigma}$ and letting $\sigma < \sigma_0 = \f 1{4 (C_0^*)^2}$, we derive that $\| ({\bf W}, W_7) \|_{\mathfrak{W}} \le 2 C_0^* \sigma \le \delta_0$. This, together with  the compatibility conditions \eqref{W6-cp}, \eqref{W5-cp}, \eqref{W123-cp}, \eqref{bc46}, \eqref{bc47}, \eqref{bc48}, yields  $({\bf W}, W_7) \in \mathfrak{W}$. We now can define the operator $\mc{T} ({\bf \hat{W}}, \hat{W}_7) = ({\bf W}, W_7)$, which maps $\mathfrak{W}$ to itself.
		
		\par {\bf Step 7. The contraction of the operator $\mathcal{T}$.}
		\par We finally verify that $\mathcal{T}$ is a contraction mapping under the following norm
		\begin{equation*}
			\| ({\bf W}, W_7) \|_{\mathfrak{U}} \co \sum_{i = 1}^6\| W_i \|_{C^{1, \al} (\overline\mn)} + \| W_7\|_{C^{2, \al} (\overline{D})},
		\end{equation*}
		where the fixed point of $\mathcal{T}$ in $\mathfrak{W}$ is exactly the solution we seek. For any two elements $(\hat{\mathbf{W}}^i, \hat{W}_7^i) \in \mathfrak{W}$ ($i=1,2$), we denote $(\mathbf{W}^i, W_7^i) = \mathcal{T} (\hat{\mathbf{W}}^i, \hat{W}_7^i)$. Set
		\begin{equation*}
			(\hat{\mathbf{W}}^d, \hat{W}_7^d) = (\hat{\mathbf{W}}^1, \hat{W}_7^1) - (\hat{\mathbf{W}}^2, \hat{W}_7^2), \ \ ({\mathbf{W}}^d, {W}_7^d) = (\mathbf{W}^1, W_7^1) - (\mathbf{W}^2, W_7^2).
		\end{equation*}
		It follows from  \eqref{6-3} that  $W_6^d$ satisfies
		\begin{equation}\label{z6-bvp-simplified}
			\begin{cases}
				\left(\partial_{y_1} + I_2(\hat{\mathbf{W}}^1, \hat{W}_7^1)\partial_{y_2} + I_3(\hat{\mathbf{W}}^1, \hat{W}_7^1)\partial_{y_3}\right)W_6^d
				= I_4^1 W_6^d + F_6 , &  \mbox{in } \mn,\\
				W_6^d(r_b,y') = G_6 (y'), &  y'\in D,
			\end{cases}
		\end{equation}
		where
		\begin{align*}
			& F_6 = \left(I_4^1 - I_4^2\right)W_6^2 - \left(I_2^1 - I_2^2\right)\partial_{y_2}W_6^2 - \left(I_3^1 - I_3^2\right)\partial_{y_3}W_6^2, \\
			& G_6 (y') =Y^-\big(r_b+\hat{W}_7^1(y'),y'\big) - Y^-\big(r_b+\hat{W}_7^2(y'),y'\big),
		\end{align*}
		with $I_j^k = I_j(\hat{\mathbf{W}}^k, \hat{W}_7^k)$ for $j=2,3,4$ and $k=1,2$.
		The explicit solution for $W_6^d$ is derived via the method of characteristics and given by
		\begin{equation}\label{z6-expr}
			\begin{split}
				W_6^d(y_1, y')
				&= \exp\left( \int_{r_b}^{y_1} I_4^1\big(t, \hat{y}_2^1(t; y_1, y'), \hat{y}_3^1(t; y_1, y')\big) dt \right)  G_6\bigg(\eta_2^1(y_1, y'), \eta_3^1(y_1, y')\bigg)  \\
				&+ \int_{r_b}^{y_1} F_6\big(s, \hat{y}_2^1(s; y_1, y'), \hat{y}_3^1(s; y_1, y')\big)
				\exp\left( \int^{y_1}_s I_4^1\big(t, \hat{y}_2^1(t; y_1, y'), \hat{y}_3^1(t; y_1, y')\big) dt \right) ds .
			\end{split}
		\end{equation}
		Here, $(t, \hat{y}_2^i(t; y_1, y'), \hat{y}_3^i(t; y_1, y'))$ represents the trajectory associated with the characteristic vector $(1, I_2^i, I_3^i)$ with $i=1,2$, and $\eta_i^1 = \hat{y}_i^1(r_b; y_1, y')$ for $i=2,3$.
		Then $W_6^d$ satisfies the estimate
		\begin{equation}\label{z6-est}
			\begin{split}
				\| W_6^d \|_{C^{1, \alpha} (\overline{\mn})} & \leq C_* \sigma \| \hat W_7^d \|_{C^{1, \alpha} (\overline{D})} + C_* \| (\hat{\mathbf{W}}^2, \hat{W}_7^2) \|_{\mathfrak{W}} \| (\hat{\mathbf{W}}^d, \hat{W}_7^d) \|_{\mathfrak{U}} \\
				& \leq C_* (\sigma + \delta_0) \| (\hat{\mathbf{W}}^d, \hat{W}_7^d) \|_{\mathfrak{U}} .
			\end{split}
		\end{equation}
		\par Next, $W_5^d $ and $W_4^d $ satisfy
		
		\begin{equation}\label{z5_eq}
			\begin{cases}
				\left(\partial_{y_1} + I_2^1 \partial_{y_2} + I_3^1 \partial_{y_3}\right)W_5^d = F_5, &\mbox{in } \mn, \\
				W_5^d(r_b,y') = G_5 (y'),& y'\in D,
			\end{cases}
		\end{equation}
		and
		\begin{equation}\label{z4_eq}
			\begin{cases}
				\left(\partial_{y_1} + I_2^1 \partial_{y_2} + I_3^1 \partial_{y_3}\right)W_4^d = F_4 , &\mbox{in } \mn, \\
				W_4^d(r_b,y') = G_4 (y'), &y'\in D,
			\end{cases}
		\end{equation}
		where
		\begin{align*}
			& F_4  =  I_6^1 - I_6^2 + (I_2^2 - I_2^1)\partial_{y_2}W_4^2 + (I_3^2 - I_3^1)\partial_{y_3}W_4^2, \\
			& F_5  = I_5^1 - I_5^2 + (I_2^2 - I_2^1)\partial_{y_2}W_5^2 + (I_3^2 - I_3^1)\partial_{y_3}W_5^2, \\
			& G_4 (y') =  b_2 \hat{z}_7(y') + R_2\big(\hat{\mathbf{W}}^1(r_b,y'),\hat{W}_7^1(y')\big) - R_2\big(\hat{\mathbf{W}}^2(r_b,y'),\hat{W}_7^2(y')\big),\\
			& G_5 (y') =  B^-\big(r_b+\hat{W}_7^1(y'),y'\big) - B^-\big(r_b+\hat{W}_7^2(y'),y'\big).
		\end{align*}
		Here$I_5^k=I_5(W_6^k,\hat{\mathbf{W}}^k,\hat{W}_7^k)$ and $I_6^k=I_6(W_6^k,\hat{\mathbf{W}}^k,\hat{W}_7^k)$ with $k=1,2$.
		By the method of characteristics and \eqref{6-11}, we obtain the following expressions
		\begin{align}
			\label{z4-final-expr-simple}
			W_4^d(y_1,y')&= \frac{b_2}{b_1} W_1^d (r_b, y') + R_4 \bigg(\hat{\mathbf{W}}^1(r_b,\eta_2^1 (y_1,y'), \eta_3^1 (y_1,y')),\hat{W}_7^1(\eta_2^1 (y_1,y'), \eta_3^1 (y_1,y'))\bigg) \notag \\
			&\quad - R_4 \bigg(\hat{\mathbf{W}}^2(r_b,\eta_2^2 (y_1,y'), \eta_3^2 (y_1,y')),\hat{W}_7^2(\eta_2^2 (y_1,y'), \eta_3^2 (y_1,y'))\bigg) \notag \\
			&\quad +\int_{r_b}^{y_1} F_4\bigg(\tau,\hat{y}_2^1(\tau;y_1,y'),\hat{y}_3^1(\tau;y_1,y')\bigg)d\tau, \\
			\label{z5-expr}
			W_5^d(y_1,y') &= G_5\left(\eta_2^1(y_1,y'),\eta_3^1(y_1,y')\right) + \int_{r_b}^{y_1} F_5\bigg(\tau, \hat{y}_2^1(\tau;y_1,y'), \hat{y}_3^1(\tau;y_1,y')\bigg) d\tau.
		\end{align}
		Moreover, $W_5^d$ satisfy the estimate
		\begin{equation}\label{est_z5_C1a}
			\begin{split}
				\| W_5^d \|_{C^{1,\alpha}(\overline{\mn})} & \le C_* \sigma \| \hat W_7^d  \|_{C^{1,\alpha}(\overline{D})} + C_* \| (\hat{\mathbf{W}}^2, \hat{W}_7^2) \|_{\mathfrak{W}} \bigg( \sum_{i =1}^6 \| \hat W_i^d \|_{C^{1,\alpha}(\overline{\mn})} + \| \hat W_7^d \|_{C^{1,\alpha}(\overline{D})} \bigg) \\
				& \leq C_* (\sigma + \delta_0) \| (\hat{\mathbf{W}}^d, \hat{W}_7^d) \|_{\mathfrak{U}}.
			\end{split}
		\end{equation}
		
		\par Set $$\tilde{\omega}_i^d = \tilde{\omega}_i^1 - \tilde{\omega}_i^2, \ \ i = 1,2,3,$$ where $(\tilde{\omega}_1^k, \tilde{\omega}_2^k, \tilde{\omega}_3^k)$ is the vorticity associated with $(\mathbf{W}^k, W_7^k)$ for $k =1,2$.
		Then it follows from \eqref{om7-co} that
		\begin{equation}\label{L1_eq}
			\begin{cases}
				\left(\partial_{y_1} + I_2^1 \partial_{y_2} + I_3^1 \partial_{y_3}\right)\tilde{\omega}_1^d + \mu^1 \tilde{\omega}_1^d = G_1^d , & \mbox{in }\mn , \\
				\tilde{\omega}_1^d(r_b,y') =R_5^d (y') ,  & y'\in D,
			\end{cases}
		\end{equation}
		where
		\begin{align*}
			& G_1^d = G_1^1 - G_1^2 + (I_2^2-I_2^1)\p_{y_2}\tilde{\omega}_1^2 + (I_3^2-I_3^1)\p_{y_3}\tilde{\omega}_1^2 + (\mu^2-\mu^1)\tilde{\omega}_1^2, \\
			& R_5^d (y') =  R_5(\hat{\mathbf{W}}^1(r_b,y'),\hat{W}_7^1(y')) - R_5(\hat{\mathbf{W}}^2(r_b,y'),\hat{W}_7^2(y')),
		\end{align*}
		with $\mu^k=\mu(\hat{\mathbf{W}}^k,\hat{W}_7^k)$, $G_1^k=G_1(\hat{\mathbf{W}}^k,\hat{W}_7^k)$ ($k=1,2$).
		Using the method of characteristics, one has
		\begin{equation}\label{L1-expr}
			\begin{split}
				\tilde{\omega}_1^d(y_1,y')  = & \exp\left( - \int_{r_b}^{y_1} \mu^1\big(s, \hat{y}(s; y_1,y')\big) ds \right) R_5^d \left(\eta_2^1(y_1,y'),\eta_3^1(y_1,y')\right) \\
				& + \int_{r_b}^{y_1} \exp\left( - \int^{y_1}_{\tau} \mu^1\big(s, \hat{y}(s; y_1,y')\big) ds \right) G_1^d\big(\tau, \hat{y}(\tau; y_1,y')\big) d\tau.
			\end{split}
		\end{equation}
		Then the following estimate holds
		\begin{equation}\label{L1-est}
			\begin{split}
				\|\tilde{\omega}_1^d\|_{C^{\alpha} (\overline{\mn})}  & \leq C_* \left( \|R_5^d \|_{C^{\alpha} (\overline{D})} + \| G_1^d \|_{C^{\alpha} (\overline{\mn})}  \right) \\
				& \leq C_* \left( \sigma + \| \tilde{\omega}_1^2 \|_{C^{1, \alpha} (\overline{\mn})} + \| (\hat{\mathbf{W}}^2, \hat{W}_7^2) \|_{\mathfrak{W}} \right) \| (\hat{\mathbf{W}}^d, \hat{W}_7^d) \|_{\mathfrak{U}} \\
				& \leq C_* (\sigma + \delta_0) \| (\hat{\mathbf{W}}^d, \hat{W}_7^d) \|_{\mathfrak{U}}.
			\end{split}
		\end{equation}
		
		Now we can turn to the estimate of $ W_i^d$ ($i=1,2,3$).  It follows from \eqref{defor-curl-Pi} that
		\begin{equation}\label{z_i_bvp}
			\begin{cases}
				d_1(y_1)\p_{y_1} W_1^d + \frac1{y_1}\p_{y_2} W_2^d + \p_{y_3} W_3^d+\frac{W_1^d}{y_1}+d_2(y_1)W_1^d =  G_0^d, & \text{in } \mn, \\
				\frac{1}{y_1}\p_{y_2} W_3^d-\p_{y_3}W_2^d + \partial_{y_1} \Pi^d =  G_2^d, & \text{in } \mn, \\
				\p_{y_3}  W_1^d- \p_{y_1} W_3^d + d_5(y_1) \p_{y_3}  W_1^d(r_b,y') + \frac{1}{y_1}\partial_{y_2} \Pi^d =  G_3^d, & \text{in } \mn,\\
				\p_{y_1}  W_2^d+\frac{ W_2^d}{y_1}-\frac{1}{y_1}\p_{y_2}  W_1^d -\frac{d_5(y_1)}{y_1}\p_{y_2} W_1^d(r_b,y') + \partial_{y_3} \Pi^d = G_4^d, & \text{in } \mn,\\
				\frac{1}{r_b^2}\p_{y_2}^2  W_1^d(r_b,y')+\p_{y_3}^2  W_1^d(r_b,y') = b_0b_1\left(\frac{1}{r_b}\p_{y_2}  W_2^d+\p_{y_3}  W_3^d\right)(r_b,y')+ q_1^d, & y'\in D, \\
				W_2^d(y_1, \pm \theta_0, y_3) = \Pi^d (y_1, \pm \theta_0, y_3) = 0,  & (y_1, y_3) \in \Sigma_{ \theta_0}^\pm,\\
				W_3^d(y_1, y_2, \pm L ) = \Pi^d (y_1, y_2, \pm L ) = 0, & (y_1, y_2) \in \Sigma_{ L}^\pm, \\
				W_1^d(r_1,y') + d_5(r_1) W_1^d(r_b,y') =  q_4^d, & y' \in D, \\
				( \frac1{r_b}\p_{y_2}W_1^d - b_0 b_1W_2^d) (r_b,\pm\theta_0,y_3) = 0 , & y_3 \in [-L,L], \\
				(\p_{y_3} W_1^d - b_0 b_1 W_3^d)(r_b,y_2,\pm L) = 0, & y_2 \in [-\theta_0,\theta_0], \\
				\Pi^d(r_b,y') =\Pi^d(r_1,y') = 0, & y'\in D,
			\end{cases}
		\end{equation}
		where
		\begin{align*}
			& \Pi^d = \Pi^1 - \Pi^2,
			\quad  G_0^d = G_0(W_5^1,W_6^1,\hat{\mathbf{W}}^1,\hat{W}_7^1) - G_0(W_5^2,W_6^2,\hat{\mathbf{W}}^2,\hat{W}_7^2), \\
			&  G_2^d = G_2(\hat{\mathbf{W}}^1,\hat{W}_7^1) - G_2(\hat{\mathbf{W}}^2,\hat{W}_7^2),
			\quad G_3^d = G_3(W_5^1,\hat{\mathbf{W}}^1,\hat{W}_7^1) - G_3(W_5^2,\hat{\mathbf{W}}^2,\hat{W}_7^2), \\
			&  G_4^d = G_4(W_5^1,\hat{\mathbf{W}}^1,\hat{W}_7^1) - G_4(W_5^2,\hat{\mathbf{W}}^2,\hat{W}_7^2), \\
			&  q_1^d = q_1(\hat{\mathbf{W}}^1(r_b,y'),\hat{W}_7^1(y')) - q_1(\hat{\mathbf{W}}^2(r_b,y'),\hat{W}_7^2(y')), \\
			&  q_4^d = q_4(\hat{\mathbf{W}}^1(r_1,y'),\hat{W}_7^1(y')) - q_4(\hat{\mathbf{W}}^2(r_1,y'),\hat{W}_7^2(y')).
		\end{align*}
		By similar analysis in \textbf{Step 5}, one has
		\begin{equation}\label{z123-est}
			\begin{split}
				\sum_{i = 1}^3 \| W_i^d \|_{C^{1, \alpha} (\overline{\mn})} & \leq C_* \bigg( \|  G_0^d \|_{C^{\alpha} (\overline{\mn})}+\sum_{i=2}^4\|  G_i^d \|_{C^{\alpha} (\overline{\mn})} + \|  q_1^d \|_{C^{0, \alpha} (\overline{D})} +  \|  q_4^d \|_{C^{1, \alpha} (\overline{D})}  \bigg)  \\
				& \leq C_* \bigg(\sigma + \sum_{i = 1}^2 \| (\hat{\mathbf{W}}^i, \hat{W}_7^i) \|_{\mathfrak{W}} \bigg) \| (\hat{\mathbf{W}}^d, \hat{W}_7^d) \|_{\mathfrak{U}} \leq C_* (\sigma + \delta_0) \| (\hat{\mathbf{W}}^d, \hat{W}_7^d) \|_{\mathfrak{U}}.
			\end{split}
		\end{equation}
		
		We proceed to estimate $W_4^d$. From \eqref{z4-final-expr-simple}, it is necessary for us to first estimate the following term:
		\begin{equation*}
			\begin{aligned}
				& \left\| \left( \hat{W}_7^1\big(\eta_2^1 (y_1,y'), \eta_3^1 (y_1,y')\big)-\hat{W}_7^1(y') \right)
				- \left( \hat{W}_7^2\big(\eta_2^2 (y_1,y'), \eta_3^2 (y_1,y')\big) - \hat{W}_7^2(y') \right) \right\|_{C^{1, \alpha} (\overline{\mn})} \\
				\leq & \left\| \hat{W}_7^d \big(\eta_2^1 (y_1,y'), \eta_3^1 (y_1,y')\big) - \hat{W}_7^d(y') \right\|_{C^{1, \alpha} (\overline{\mn})} \\
				&\quad+ \left\| \hat{W}_7^2\big(\eta_2^1 (y_1,y'), \eta_3^1 (y_1,y')\big) -  \hat{W}_7^2\big(\eta_2^2 (y_1,y'), \eta_3^2 (y_1,y')\big) \right\|_{C^{1, \alpha} (\overline{\mn})}\\
				\leq &
				C_* \sum_{i = 2}^3 \left\| \eta_i^1 - y_i \right\|_{C^{1, \alpha} (\overline{\mn})} \left\| \hat{W}_7^d  \right\|_{C^{2, \alpha} (\overline{D})} + \sum_{i = 2}^3 \left\| \eta_i^1 - \eta_i^2 \right\|_{C^{1, \alpha} (\overline{\mn})} \left\| \hat{W}_7^2 \right\|_{C^{2, \alpha} (\overline{D})}.
			\end{aligned}
		\end{equation*}
		Let $\Gamma_i (\tau; y_1,y') = \hat{y}_i^1 (\tau; y_1,y')- \hat{y}_i^2 (\tau; y_1,y') $ and $\Gamma_i (r_b;y_1,y' ) = \eta_i^1 (y_1,y') - \eta_i^2 (y_1,y')$ for $i=2,3$. Then we have
		\begin{equation}\label{gamma_int_form_simple}
			\begin{aligned}
				\Gamma_2 (\tau; y_1,y') &= \int_{y_1}^{\tau} b_{22} (s;y_1,y') \Gamma_2 (s; y_1,y') + b_{23} (s; y_1,y') \Gamma_3 (s; y_1,y') ds + \int_{y_1}^{\tau} b_2  (s; y_1,y') ds, \\
				\Gamma_3 (\tau; y_1,y') &= \int_{y_1}^{\tau} b_{32} (s; y_1,y') \Gamma_2 (s; y_1,y') + b_{33} (s; y_1,y') \Gamma_3 (s; y_1,y') ds + \int_{y_1}^{\tau} b_3  (s; y_1,y') ds,
			\end{aligned}
		\end{equation}
		where $b_{ij}(s;y_1,y') \ (i, j = 2, 3)$ are functions of $\hat{\mathbf{W}}^1$ and $\hat{W}_7^1$, and $b_i(s;y_1,y') \ (i = 2, 3)$ are functions of $\hat{\mathbf{W}}^d$ and $\hat{W}_7^d$. It follows from Gronwall inequality that
		\begin{equation*}
			\sum_{i = 2}^3 \left\| \eta_i^1 - \eta_i^2 \right\|_{C^0 (\overline{\mn})} \leq C_* \bigg(  \sum_{i = 1}^6 \left\| \hat{W}_i^d \right\|_{C^0 (\overline{\mn})} + \| \hat{W}_7^d \|_{C^1 (\overline{D})}\bigg).
		\end{equation*}
		Similarly, one has
		\begin{equation*}
			\sum_{i = 2}^3 \left\| \eta_i^1 - \eta_i^2 \right\|_{C^{1, \alpha} (\overline{\mn})} \leq C_* \| (\hat{\mathbf{W}}^d, \hat{W}_7^d) \|_{\mathfrak{U}}.
		\end{equation*}
		Combining the above estimates leads to
		\begin{equation}\label{z4-est-final}
			\begin{split}
				\| W_4^d \|_{C^{1, \alpha} (\overline{\mn})}  & \leq C_* \left( \| W_1^d (r_b, \cdot) \|_{C^{1, \alpha} (\overline{D})} + \| R_4 (\hat{\mathbf{W}}^1,\hat{W}_7^1 ) - R_4 (\hat{\mathbf{W}}^2,\hat{W}_7^2 ) \|_{C^{1, \alpha} (\overline{D})} + \| F_4 \|_{C^{1, \alpha} (\overline{\mn})}  \right) \\
				& \leq C_* (\sigma + \delta_0)  \| (\hat{\mathbf{W}}^d, \hat{W}_7^d)  \|_{\mathfrak{U}}.
			\end{split}
		\end{equation}
		\eqref{6-2} implies that
		\begin{equation*}
			W_7^d(y') = \frac{1}{b_{1} }W_1^d(r_b,y')-\frac{1}{b_{1} }  \left( R_1(\h{\bf W}^1(r_b,y'),\h W_7^1(y')) - R_1(\h{\bf W}^2(r_b,y'),\h W_7^2(y')) \right), \quad y' \in D.
		\end{equation*}
		Thus there holds
		\begin{equation}\label{z7-est}
			\begin{split}
				\| W_7^d \|_{C^{1, \alpha} (\overline{D})} & \leq C_* \left( \| W_1^d (r_b, \cdot) \|_{C^{1, \alpha} (\overline{D})}  + \| R_1(\h{\bf W}^1(r_b,y'),\h W_7^1(y')) - R_1(\h{\bf W}^2(r_b,y'),\h W_7^2(y')) \|_{C^{1, \alpha} (\overline{D})}  \right) \\
				& \leq C_* (\sigma + \delta_0)  \|(\hat{\mathbf{W}}^d, \hat{W}_7^d) \|_{\mathfrak{U}}.
			\end{split}
		\end{equation}
		It follows from \eqref{pf36} that
		\begin{equation*}
			\begin{cases}
				\p_{y_2} W_7^d (y') = b_0 W_2^d (r_b, y') + g_2 (\hat{{\bf W}}^1 (r_b, y'), \hat{W}_7^1 (y')) - g_2 (\hat{{\bf W}}^2 (r_b, y'), \hat{W}_7^2 (y')), & \mbox{in } D, \\
				\p_{y_3} W_7^d  (y') = b_0 W_3^d  (r_b, y') + g_3 (\hat{{\bf W}}^1 (r_b, y'), \hat{W}_7^1 (y')) - g_3 (\hat{{\bf W}}^2 (r_b, y'), \hat{W}_7^2 (y')), & \mbox{in } D,
			\end{cases}
		\end{equation*}
		from which one gets
		\begin{equation}\label{z7-est-2}
			\begin{aligned}
				&\| \p_{y_2} W_7^d \|_{C^{1, \alpha} (\overline{D})} +  \| \p_{y_3}W_7^d \|_{C^{1, \alpha} (\overline{D})}\\
				& \leq C_* \sum_{i = 2}^3 \left( \| W_i^d (r_b, \cdot) \|_{C^{1, \alpha} (\overline{D})} + \| g_i  (\hat{{\bf W}}^1 (r_b, y'), \hat{W}_7^1 (y')) - g_i (\hat{{\bf W}}^2 (r_b, y'), \hat{W}_7^2 (y')) \|_{C^{1, \alpha} (\overline{D})} \right) \\
				& \leq C_* (\sigma + \delta_0) \| (\hat{\mathbf{W}}^d, \hat{W}_7^d)\|_{\mathfrak{U}}.
			\end{aligned}
		\end{equation}
		
		Combining \eqref{z6-est}, \eqref{est_z5_C1a},  \eqref{z123-est}, \eqref{z4-est-final}, \eqref{z7-est} and \eqref{z7-est-2} yields
		\begin{equation*}
			\| ({\mathbf{W}}^d, {W}_7^d)  \|_{\mathfrak{U}} \leq C_0^* (\sigma + \delta_0) \| (\hat{\mathbf{W}}^d, \hat{W}_7^d) \|_{\mathfrak{U}}.
		\end{equation*}
		Since $\delta_0 = \sqrt{\sigma}$, setting $\sigma < \sigma_0 = \frac{1}{16 (C_0^*)^2}$ then $\| ({\mathbf{W}}^d, {W}_7^d)  \|_{\mathfrak{U}} \leq \frac{1}{2} \| (\hat{\mathbf{W}}^d, \hat{W}_7^d) \|_{\mathfrak{U}}$. Thus $\mathcal{T}$ is a contraction mapping with respect to the norm $\|\cdot\|_{\mathfrak{U}}$. Consequently, there exists a unique fixed point $({\bf {W}}, {W}_7) \in \mathfrak{W}$ such that $\mc{T} ({\bf {W}}, {W}_7) = ({\bf W}, W_7) $.

		Finally, we prove that the auxiliary function $\Pi$ associated with the fixed point $({\bf {W}}, {W}_7)$ satisfies $\Pi \equiv 0$ in $\mn$. It follows from the definition of $G_i$ that
		\begin{equation}
			\begin{cases}
				\begin{aligned}
					& - \p_{y_1} \Pi = J_2^{W_7} W_3 - J_3^{W_7} W_2 - \tilde{\om}_1,  \\
					&- \frac1{y_1}\p_{y_2} \Pi = J_3^{W_7} W_1 - J_1^{W_7} W_3 -\frac {W_2 \e\omega_1 + J_3^{W_7}W_5}{\bar U(J_0^{W_7})+W_1}+\frac{e^{( S_0+W_4)}\e\rho^{\gamma-1}}{(\gamma-1)(\bar U(J_0^{W_7})+W_1)} J_3^{W_7}W_4, \\
					&- \p_{y_3} \Pi = J_1^{W_7} W_2 +\frac{W_2}{J_0^{W_7}}- J_2^{W_7} W_1 - \frac {W_3 \e\omega_1 - J_2^{W_7}W_5}{\bar U(J_0^{W_7})+W_1}-\frac{e^{( S_0+W_4)}\e\rho^{\gamma-1}}{(\gamma-1)(\bar U(J_0^{W_7})+W_1)} J_2^{W_7}W_4.
				\end{aligned}
			\end{cases}
		\end{equation}
		These, together with \eqref{om7}, show that
		\begin{equation*}
			J_1^{W_7} J_2^{W_7} - J_2^{W_7} J_1^{W_7}=-\frac{J_2^{W_7}}{J_0^{W_7}}, \, \ \
			J_2^{W_7} J_3^{W_7} = J_3^{W_7} J_2^{W_7}, \, \ \
			J_1^{W_7} J_3^{W_7} = J_3^{W_7} J_1^{W_7}.
		\end{equation*}
		Therefore, one gets
		\begin{equation*}
			- J_1^{W_7} (\p_{y_1} \Pi) - J_2^{W_7} \left(\frac1{y_1}\p_{y_2} \Pi\right) - J_3^{W_7} (\p_{y_3} \Pi)-\frac{\p_{y_1} \Pi}{J_0^{W_7}} = 0, \q \text{in } \mn.
		\end{equation*}
		Note that $\|W_7 \|_{C^{3, \alpha} (\overline{D})} \leq \delta_0$. For sufficiently small $\delta_0$, $\Pi$ satisfies a second-order uniformly elliptic system without zero-order terms. Moreover, combined with the boundary conditions for $\Pi$ on $\partial \mn$, we conclude that $\Pi \equiv 0$ in $\mn$ by employing the maximum principle and Hopf's lemma. Therefore, $({\bf {W}}, {W}_7)$ is the desired solution, and we thereby complete the proof of Theorem \ref{main}.
		
		\par {\bf Acknowledgement.}   The research of Zihao Zhang is partially supported by  the Postdoctoral Fellowship Program of the China Postdoctoral Science Foundation (Grant No. GZB20250719) and the China Scholarship Council (Grant No. 202606170019).
		\par {\bf Data availability.} No data was used for the research described in the article.
		\par {\bf Conflict of interest.} This work does not have any conflicts of interest.

	\end{document}